\documentclass[10pt]{amsart}
\usepackage[margin=1.5in]{geometry}
\newtheorem{theorem}{Theorem}[section]
\newtheorem{lemma}[theorem]{Lemma}
\newtheorem{proposition}{Proposition}[section]
\newtheorem{corollary}{Corollary}[section]
\theoremstyle{definition}

\theoremstyle{remark}

\numberwithin{equation}{section}

\begin{document}

\title[Polyatomic BGK model near Maxwellians]{Ellipsoidal BGK model for polyatomic molecules near Maxwellians: A dichotomy in the dissipation estimate}

\author{ SEOK-BAE YUN }
\address{Department of Mathematics, Sungkyunkwan University, Suwon 440-746, Republic of Korea}
\email{sbyun01@skku.edu}



\thanks{This research was supported by Basic Science Research Program through the National Research Foundation of Korea(NRF) funded by the Ministry of Education(NRF-2016R1D1A1B03935955)}
\begin{abstract}
We consider the global existence and asymptotic behavior of classical solutions to the ellipsoidal BGK model for polyatomic molecules when the initial data starts sufficiently close to a global polyatomic Maxwellian. We observe that the linearized relaxation operator is decomposed into a trully polyatomic part and a essentially monatomic part, leading to a dichotomy in the dissipative property in the sense that the degeneracy of the dissipation shows an abrupt jump as the relaxation parameter $\theta$ reaches zero. Accordingly,
we employ two different sets of micro-macro system  to
derive the full coercivity and close the energy estimate.
\end{abstract}
\maketitle
\section{introduction}

 The collective dynamics of rarefied gases at the mesoscopic scale is described by the celebrated Boltzmann equation.
But the practical application of the Boltzmann equation has been restricted by its highly resource-consuming features such as the complicated structure of the collision operator, high dimensionality and stiffness problem.
In this regard, Bhatnagar, Gross, Krook \cite{BGK} and, independently Welander \cite{Wel}, suggested a model equation by
replacing the collision operator with a relaxation operator
which still keeps the most important features of the Boltzmann equation such as the conservation laws, $H$-theorem and the correct hydrodynamic limit to the Euler equation. Ever since it was introduced,  the BGK model has been widely used in place of the Boltzmann equation because it reproduces the qualitative features of the Boltzmann dynamics very well
at much lower computational costs.

Both the Boltzmann equation and the BGK  model are derived under the assumption that the gas consists of monatomic molecules. 
The necessity of kinetic equations that account for the collisional dynamics of polyatomic molecules is apparent, considering that
there are very few elements in the nature which stay stable as monatomic molecules at room temperature.
Any attempt for the description of kinematics of polyatomic moleculess, however, must allow some simplifying assumptions or phenomenological description because the diversity of the inner configuration of polyatomic molecules make it almost impossible to express the pre-post collision process in an explicit form, except for some special cases.
One such formulation is so-called the internal energy formulation where a new variable $I$ is introduced to incorporate the information on the non-translational internal energy due to the
molecular structure \cite{ABLP,ALPP,BL,BDLP,Brull3,Cai,KAG,PRS,RS,RS2,TFA}.

In this paper, we study the existence and asymptotic behavior for the polyatomic ellipsoidal BGK model, which is a polyatomic generalization  of the original BGK model using such internal energy formulation: \cite{ABLP,ALPP}:
\begin{align}\label{polyatomic ESBGK}
\begin{split}
\partial_tF+v\cdot \nabla_xF&=A_{\nu,\theta}(\mathcal{M}_{\nu,\theta}(F)-F),\cr
\qquad F(0,x,v,I)&=F_0(x,v,I).
\end{split}
\end{align}
The velocity-energy distribution function $F(t,x,v,I)$ represents the number density on phase point $(x,v)\in \mathbb{T}^3_x\times \mathbb{R}^3_v$ with non-translational internal energy $I^{2/\delta}~(I\geq0)$ at time $t\geq0$.
The parameter $\delta>0$ measures the degree of excitation of non-translational mode of the molecules such as the rotational or vibrational mode.
The collision frequency $A_{\nu,\theta}$ is given by
$A_{\nu,\theta}=(\rho^{\alpha}T^{\beta}_{\delta})/(1-\nu+\theta\nu)$
for some $0\leq \alpha,\beta\leq 1$. ($\rho$ and $T_{\delta}$ are defined below.) Throughout this paper, we fix $\alpha=\beta=1$ for simplicity.
The relaxation parameters $-1/2<\nu<1$ and $0\leq\theta\leq1$ are introduced to
reproduce the correct Prandtl number and the second viscosity coefficient in the Chapman-Enskog expansion \cite{ALPP}.
The number $1/\theta$ is interpreted as the relaxation collision number, which is the average number of collisions needed
to transfer the rotational and vibrational internal energy into the translational energy \cite{ABLP,Cai}.

We define the macroscopic density, momentum, stress tensor and total energy by
\begin{align*}
\rho(t,x)&=\int_{\mathbb{R}^3 \times \mathbb{R}^+} F(t,x,v,I) dvdI,\cr
U(t,x)&=\frac{1}{\rho}\int_{\mathbb{R}^3 \times \mathbb{R}^+} vF(t,x,v,I) dvdI, \cr
\Theta(t,x) &= \frac{1}{\rho}\int_{\mathbb{R}^3 \times \mathbb{R}^+} (v-U)\otimes(v-U)F(t,x,v,I) dvdI, \cr
E(t,x)&=\int_{\mathbb{R}^3 \times \mathbb{R}^+} \left(\frac{1}{2}|v|^2 +I^{\frac{2}{\delta}}\right)F(t,x,v,I) dvdI.
\end{align*}
The total energy is decomposed further into the following three parts:
\[
E=E_{kin}+E_{tr}+E_{I,\delta}
\]
where the kinetic energy $E_{kin}$, the internal energy
due to the translational motion $E_{tr}$, and the internal energy attributed to the internal configuration of the molecules
$E_{I,\delta}$ are given respectively by
\begin{align*}
E_{kin}&=\frac{1}{2}\rho |U|^2,\cr
E_{tr}&=\frac{1}{2}\int_{\mathbb{R}^3 \times \mathbb{R}^+} |v-U|^2 F(t,x,v,I) dvdI, \cr
E_{I,\delta}&=\int_{\mathbb{R}^3 \times \mathbb{R}^+}  I^{\frac{2}{\delta}} F(t,x,v,I) dvdI.
\end{align*}
We also define the total internal energy $E_{\delta}$:
\begin{align*}
E_{\delta}(t,x)&=E_{tr}+E_{I,\delta}\cr
&=\int_{\mathbb{R}^3 \times \mathbb{R}^+} \left(\frac{1}{2}|v-U|^2 +I^{\frac{2}{\delta}}\right)F(t,x,v,I) dvdI,
\end{align*}
from which we can define the corresponding temperatures $T_{\delta}$, $T_{tr}$ and $T_{I,\delta}$ using the equipartition principle:
\begin{align*}
E_{\delta}&=\frac{3+\delta}{2}\rho T_{\delta},\quad E_{tr}=\frac{3}{2}\rho T_{tr}, \quad E_{I,\delta}= \frac{\delta}{2}\rho T_{I,\delta}.
\end{align*}
Consequently, $T_{\delta}$ is represented by a convex combination of $T_{tr}$ and $T_{I,\delta}$:
\begin{eqnarray*}
T_{\delta}=\frac{3}{3+\delta}T_{tr}+\frac{\delta}{3+\delta}T_{I,\delta}.
\end{eqnarray*}
For $-1/2<\nu<1$ and  $0\leq\theta\leq1$, we define the  relaxation  temperature $T_{\nu,\theta}$ and the corrected  temperature tensor $\mathcal{T}_{\theta}$ by
\begin{align*}
T_{\theta}&=\theta T_{\delta}+(1-\theta)T_{I,\delta},\cr
\mathcal{T}_{\nu,\theta}&=\theta T_{\delta}Id+(1-\theta)\big\{(1-\nu)T_{tr}Id+\nu \Theta\big\}.
\end{align*}
Now, the polyatomic ellipsoidal Maxwellian $\mathcal{M}_{\nu,\theta}$ reads
\begin{eqnarray}\label{poly maxwellian}
\mathcal{M}_{\nu,\theta}(F)= \frac{\rho \Lambda_{\delta}}{\sqrt{\det (2\pi \mathcal{T}_{\nu,\theta}}) \,T_{\theta}^{\frac{\delta}{2}}}\exp\left(-\frac{1}{2}(v-U)^{\top}\mathcal{T}^{-1}_{\nu,\theta}(v-U)-\frac{I^{\frac{2}{\delta}}}{T_{\theta}}\right),
\end{eqnarray}
where $\Lambda_{\delta}$ is the normalizing factor: $\Lambda_{\delta}= 1/\int_{\mathbb{R}_+} e^{-I^{2/\delta}} dI$.

The relaxation operator satisfies the following cancellation property:
\begin{eqnarray*}
\int_{\mathbb{R}^3 \times \mathbb{R}^+} (\mathcal{M}_{\nu,\theta}(F)-F)
\left(
\begin{array}{c}1\cr v\cr\frac{1}{2}|v|^2 +I^{\frac{2}{\delta}}\end{array}\right)
 dvdI = 0,
\end{eqnarray*}
which leads to the conservation of mass, momentum and energy:
\begin{align}\label{ConservationLawsF}
\begin{split}
\int F(t)dxdvdI
&=\int F_0dxdvdI,\cr
\int F(t)vdxdvdI
&=\int F_0vdxdvdI,\cr
\int F(t)\left(\frac{|v|^2}{2}+I^{2/\delta}\right)dxdvdI
&=\int F_0\left(\frac{|v|^2}{2}+I^{2/\delta}\right)dxdvdI.
\end{split}
\end{align}
The $H$-theorem for this model was established in \cite{ALPP}
(See also \cite{Brull3,PY2}):
\[
\frac{d}{dt}\int_{\mathbb{R}^3 \times \mathbb{R}^+}F(t)\ln F(t)dvdI\leq 0.
\]
\newline

In this paper, we study the dynamics of the polyatomic BGK model (\ref{polyatomic ESBGK}) near a global polyatomic Maxwellian:
\begin{align}\label{poly maxwellian}
m(v,I)=\frac{\Lambda_{\delta}}{\sqrt{(2\pi)^3}}e^{-\frac{|v|^2}{2}-I^{2/\delta}}.
\end{align}
For this, we define the perturbation $f$ around the equilibrium by
\begin{align}\label{perturbation}
F=m+\sqrt{m}f,\quad F_0=m+\sqrt{m}f_0
\end{align}
and rewrite (\ref{polyatomic ESBGK}) as
\begin{align*}
\partial_t f+v\cdot\nabla_xf&=L_{\nu,\theta}f+\Gamma_{\nu,\theta}(f),\cr
f(0,x,v,I)&=f_0(x,v,I),
\end{align*}
where $L_{\nu,\theta}$ denotes the linearized relaxation operator and $\Gamma_{\nu,\theta}(f)$ is the nonlinear perturbation.
(See Section 2.) We then analyze this linearized polyatomic BGK model in the framework of nonlinear energy methods
developed in, for example, \cite{Guo whole,Guo VMB, Guo VPB}.

The most important step is to verify the dissipative nature of the linearized relaxation operator $L_{\nu,\theta}$.
In this regard, we make a key observation that there exists a dichotomy in the coercive estimate of $L_{\nu,\theta}$ (See Section 3.):
\[
-(1-\nu+\theta\nu)\langle L_{\nu,\theta}f,f\rangle_{L^2_{v,I}}\geq\theta \|(I-P_p)f\|^2_{L^2_{v,I}}, \quad (0<\theta\leq 1)
\]
and
\[
-(1-\nu)\langle L_{\nu,0}f,f\rangle_{L^2_{v,I}}\geq\left(1-|\nu|\right) \|(I-P_m)f\|^2_{L^2_{v,I}}. \quad (\theta=0)
\]
Note that the coefficient  in the l.h.s of the above dissipative estimates change continuously as $\theta$ goes to $0$, while the coefficient in the right hand side jump from $\theta$ to $1-|\nu|$ at $\theta=0$. More importantly,
the macroscopic projection on the right hand side, which determines the degeneracy of the dissipation, changes abruptly from
the projection $P_p$ on
\begin{eqnarray*}
span\bigg\{\sqrt{m },~v\sqrt{m },~\frac{(|v|^2-3)+(2I^{2/\delta}-\delta)}{\sqrt{2(3+\delta)}}\sqrt{m}\bigg\}, \quad (0<\theta\leq 1)
\end{eqnarray*}
to the projection $P_m$ on
\begin{eqnarray*}
span\bigg\{\sqrt{m },~v\sqrt{m },~\frac{|v|^2-3}{\sqrt{6}}\sqrt{m},~ \frac{I^{2/\delta}-\delta}{\sqrt{2\delta}}\sqrt{m }\bigg\}.
\qquad (\theta=0)
\end{eqnarray*}
Therefore, the degeneracy at $\theta=0$ is strictly stronger than the non zero $\theta$ case.

This agrees well with the similar dichotomy in the nonlinear entropy-entropy production estimate
observed in \cite{PY2},  of which the above estimates can be considered as a linearized version:
\begin{align*}
D_{\nu,\theta}(f)&\geq\theta A_{\nu,\theta} H(f|\mathcal{M}_{0,1}),\hspace{1cm}(0<\theta \leq1),
\end{align*}
and
\begin{align*}
D_{\nu,0}(f)&\geq\min\{1-\nu, 1+2\nu\} A_{\nu,0} H(f|\mathcal{M}_{0,0}).\quad (\theta=0)
\end{align*}
Here, $D_{\nu,\theta}$ and $H(f|g)$ denote the entropy production functional and the relative entropy  for (\ref{polyatomic ESBGK}) respectively. See \cite{PY2} for the exact definition of the target polyatomic Maxwellians $\mathcal{M}_{0,1}$ and $\mathcal{M}_{0,0}$.
In \cite{PY2}, however, it is not clear whether such dichotomy is an intrinsic property of the model, or
 can be resolved into a better estimate that interpolates the two entropy production estimates.

It is explicitly shown in Section 3 that the linearized
relaxation operator $L_{\nu,\theta}$ is divided into a truly polyatomic part and an essentially monatomic-like part. In the case $0<\theta\leq 1$, the dissipation is governed by the former, whereas the dissipation for  $\theta=0$ case is governed by the latter. This shows that such dichotomy is intrinsic and cannot be avoided by developing a refined argument.

Recalling that  $\theta^{-1}$ is interpreted as the average number of collisions needed for the non-translational energy
due to the molecular configuration to be transferred, we see that such dichotomy has a nice physical interpretation: when $\theta=0$, the relaxation collision number is infinite, and therefore, no matter how many collisions occur, the exchange between the translational energy and the non-translational energy does not happen, making the kinematics essentially - though not exactly -  that of the monatomic gases. (Note that in the kernel of $P_m$, the translational energy and the non-translational energy are completely split, whereas they are given in an entangled form in the kernel of $P_p$)
We, however, menstion that such physical interpretation alone does not give any hint that there has to be a discontinuity at $\theta=0$.

As a result of such dichotomy, we need to employ two different types of micro-macro decomposition, namely, the polyatomic decomposition:
\[
f=P_pf+(I-P_p)f\quad (0<\theta\leq 1),
\]
and the monotomic-like decomposition:
\[
f=P_mf+(I-P_m)f,\quad (\theta=0).
\]
Therefore, we need to study two different sets of micro-macro equations accordingly, in order to fill up the degeneracy and to derive the full coercivity.

\subsection{Main result}
%
%
%
%
We define the high-order energy functional $\mathcal{E}(f(t))$ :
\begin{eqnarray*}
\mathcal{E}\big(f(t)\big)
=\frac{1}{2}\sum_{|\alpha|+|\beta|\leq N}\|\partial^{\alpha}_{\beta}f(t)\|^2_{L^2_{x,v,I}}+\sum_{|\alpha|+|\beta|\leq N}\int^t_0\|\partial^{\alpha}_{\beta}f(s)\|^2_{L^2_{x,v,I}}ds.
\end{eqnarray*}
\begin{theorem}\label{main theorem}
Let $-1/2<\nu<1$, $0\leq\theta\leq1$ and $N\geq 4$. Suppose that $F_0=m+\sqrt{m}f_0\geq 0$ has the same mass, momentum and energy with $m$:
\begin{eqnarray}\label{ConservationLawsf0}
\begin{split}
\int_{\mathbb{T}^3_x\times\mathbb{R}^3_v\times\mathbb{R}^+_I}f_0\sqrt{m} ~dxdvdI&=0,\cr
\int_{\mathbb{T}^3_x\times\mathbb{R}^3_v\times\mathbb{R}^+_I}f_0v\sqrt{m} ~dxdvdI&=0,\cr
\int_{\mathbb{T}^3_x\times\mathbb{R}^3_v\times\mathbb{R}^+_I}f_0\left\{\frac{1}{2}|v|^2+I^{2/\delta}\right\}\sqrt{m} ~dxdvdI&=0.
\end{split}
\end{eqnarray}
 Then there exist  $\varepsilon>0$  and $C=C(f_0,N,\nu,\theta,\delta)>0$,
 such that if $\mathcal{E}(0)<\varepsilon$, then there exists a unique global in time solution $f$ for (\ref{LBGK}) satisfying
 \begin{enumerate}
 \item The distribution function is non-negative for all $t\geq 0$:
 \[
 F=m+\sqrt{m}\,f\geq 0,
 \]
 and satisfies the conservation laws (\ref{ConservationLawsf}).
 \item The high-order energy functional is uniformly bounded:
 \[
 \mathcal{E}(t)\leq C\mathcal{E}(0).
 \]
\item The initial perturbation decays exponentially fast:
\[
\sum_{|\alpha|+|\beta|\leq N}\|\partial^{\alpha}_{\beta}f(t)\|^2_{L^2_{x,v,I}}\leq C e^{-Ct}.
\]
\end{enumerate}
\end{theorem}

A brief review on the related literature is in order. We start with the original monatomic BGK model.
The first mathematical study of the BGK model was made in \cite{Perthame} where Perthame established the existence of weak solutions under the assumption of finite mass, energy and entropy.
Perthame and Pulvirenti then studied the existence of unique mild solutions
in a weighted $L^{\infty}$ space in \cite{PP}. These results were  extended, for example, to
Cauchy problem for $L^p$ data \cite{WZ}, plasma \cite{Zhang} or gases under the influence of external forces \cite{ZH}.
Ukai studied the stationary problem in a bounded interval with a fixed boundary condition in \cite{Ukai-BGK}.
For the application of BGK type models to various macroscopic limits, see \cite{B,DMOS,LT,MMM,Mellet,SR1,SR2}.
The existence of classical solutions and their asymptotic behavior were studied in \cite{Bello,Chan,Yun}.
Some error analysis for numerical schemes for BGK model can be found in  \cite{Issautier,RSY}.

Recently, the interest on the ES-BGK model, which is a generalized version of the monatomic BGK model designed to reproduce the physical Prandtl number, revived after the $H$-theorem was verified for this model in \cite{ALPP}. (See also \cite{Brull2,Yun4}). For the existence results of this model in various situations, see \cite{DWY, PY1,Yun2,Yun3}.

The study of the ellipsoidal BGK model for polyatomic particles started in \cite{ALPP} is in its initial stage. The $H$-theorem was shown to hold in \cite{ALPP,Brull3}. Entropy-entropy production estimate for this model was established in \cite{PY2}, where the dichotomy in the entropy dissipation mechanism mentioned above, was first observed. The extension of \cite{PP} arguments to the polyatomic case was made in \cite{PY}. In the near-equilibrium regime, no existence result is available so far.

We mention that there has been an alternative approach besides the internal energy formulation to construct BGK type model for polyatomic molecules< where the polyatomic gas is treated as a mixture of monatomic gases endowed with discrete levels of internal energy \cite{GS,GS2}.

We omit the reference review on the numerical results on  BGK type models (monatomic or polyatomic), since they are huge. Interested readers may refer to \cite{ABLP,ALPP,Cai,GT,Issautier,KAG,PC,PPuppo,RS,RSY} and references therein. For general review
on the mathematical and physical theory of kinetic equations, see \cite{C,CIP,GL,Krem,Sone,Sone2,Stru,UT,V}.
\newline

The followings are the notations and conventions kept throughout this paper:
\begin{itemize}
\item All the constants, usually denoted by $C$ will be defined generically.
\item For $\kappa\in\mathbb{R}^3$, $\kappa^{\top}$ denotes its transpose.
\item For symmetric $n\times n$ matrices $A$ and $B$, $A\leq B$ means that $B-A$ satisfies $k^{\top}\big\{B-A\big\}k\geq0$ for all $k\in \mathbb{R}^n$.

\item When there's no risk of confusion, we  use $\mathcal{E}(t)$ instead of $\mathcal{E}\big(f(t)\big)$ for simplicity. The latter notation will be employed when the dependency needs to be clarified.
\item We slightly abuse the notation to define the summation on the index set $i<j$ by
\[
\sum_{i<j}a_{ij}=a_{12}+a_{23}+a_{31}.
\]
\item  $\langle\cdot,\cdot\rangle_{L^2_{v,I}}$ and $\langle\cdot,\cdot\rangle_{L^2_{x,v,I}}$ denote the standard $L^2$ inner product on
$\mathbb{R}^3_v\times\mathbb{R}^{+}_{I}$ and $\mathbb{T}^3_x\times \mathbb{R}^3_v\times\mathbb{R}^{+}_{I}$ respectively:
\begin{align*}\begin{split}
\langle f,g\rangle_{L^2_{v,I}}&=\int_{\mathbb{R}^3\times\mathbb{R}^+}f(v,I)g(v,I)dvdI,\cr
\langle f,g\rangle_{L^2_{x,v,I}}&=\int_{\mathbb{T}^3\times \mathbb{R}^3\times\mathbb{R}^+}f(x,v,I)g(x,v,I)dxdvdI.
\end{split}\end{align*}
\item  $\|\cdot\|_{L^2_{v,I}}$ and $\|\cdot\|_{L^2_{x,v,I}}$ denote the standard $L^2$ norms on
$\mathbb{R}^3_v\times\mathbb{R}^{+}_{I}$ and $\mathbb{T}^3_x\times \mathbb{R}^3_v\times\mathbb{R}^{+}_{I}$ respectively:
\begin{align*}
\|f\|_{L^2_{v,I}}&=\Big(\int_{\mathbb{R}^3\times\mathbb{R}^+}|f(v,I)|^2dvdI\Big)^{\frac{1}{2}},\cr
\|f\|_{L^2_{x,v,I}}&=\Big(\int_{\mathbb{T}^3\times \mathbb{R}^3\times\mathbb{R}^+}|f(x,v,I)|^2dxdvdI\Big)^{\frac{1}{2}}.
\end{align*}
\item We use the following notations for the multi-indices and differential operators:
\begin{eqnarray*}
\alpha=[\alpha_0,\alpha_1,\alpha_2,\alpha_3,\alpha_4],\quad \beta=[\beta_1,\beta_2,\beta_3],
\end{eqnarray*}
and
\begin{align*}
\partial^{\alpha}_{\beta}&=\partial^{\alpha_0}_t\partial^{\alpha_1}_{x_1}\partial^{\alpha_2}_{x_2}\partial^{\alpha_3}_{x_3}
\partial^{\alpha_4}_I
\partial^{\beta_1}_{v_1}\partial^{\beta_2}_{v_2}\partial^{\beta_3}_{v_3}.
\end{align*}
\end{itemize}

The paper is organized as follows: In Section 2, we consider the linearization of the relaxation operator.  Then section 3 is devoted to the coercivity estimate for the linearized relaxation operator.
We treat the case $0<\theta\leq1$ and $\theta=0$ separately, yielding different dissipation estimate in each case.
In Section 4, we derive various estimates for macroscopic fields.
In Section 5, we consider the existence of the local in time classical solution. Section 6 is devoted to
the study of the micro-macro systems, where, due to the dichotomy observed in Section 4, the case $0<\theta\leq1$ and $\theta=0$ are considered separately. Finally, we prove the main result in Section 7.
%
%
%
%
%
%
%
%
%
%
\section{linearization of the polyatomic BGK model}
In this section, we carry out the linearization of (\ref{polyatomic ESBGK}) around the normalized global polyatomic Maxwellian (\ref{poly maxwellian}).
%
\subsection{Transitional fields:} Let $F_{\eta}$ denote the transition from the solution $F$ of (\ref{polyatomic ESBGK}) to the global polyatomic Maxwellian $m$:
\[
F_{\eta}=\eta F+(1-\eta)m=m+\eta f\sqrt{m}.\qquad (0\leq\eta\leq1)
\]
where $f$ is defined in (\ref{perturbation}).
In view of the following identities:
\begin{align*}
&\rho=\int_{\mathbb{R}^3\times \mathbb{R}_+}FdvdI,
\quad\rho U=\int_{\mathbb{R}^3\times \mathbb{R}_+}FvdvdI,\cr
&\rho\mathcal{T}_{\nu,\theta}+\frac{\theta}{3+\delta}\rho|U|^2Id+(1-\theta)\left\{\frac{1-\nu}{3}\rho|U|^2Id+\nu\rho U\otimes U\right\}\cr
&\hspace{2cm}=\theta\left\{\int_{\mathbb{R}^3\times \mathbb{R}_+}F\left(\frac{1}{3+\delta}|v|^2+\frac{2}{3+\delta}I^{2/\delta}\right)
dvdI\right\}Id\cr
&\hspace{2cm}+(1-\theta)\left\{\int_{\mathbb{R}^3\times \mathbb{R}_+}F\left(\frac{1-\nu}{3}|v|^2Id+\nu v\otimes v\right)dvdI\right\},\cr
&\rho T_{\theta}+\frac{\theta}{3+\delta}\rho|U|^2
=\theta\left\{\int_{\mathbb{R}^3\times \mathbb{R}_+}F\left(\frac{1}{3+\delta}|v|^2+\frac{2}{3+\delta}I^{2/\delta}\right)dvdI\right\}\cr
&\hspace{2.7cm}+(1-\theta)\left\{\frac{2}{\delta}\int_{\mathbb{R}^3\times \mathbb{R}_+}FI^{2/\delta}dvdI\right\},
\end{align*}
we define transitional macroscopic fields:
$\rho_{\eta}$, $U_{\eta}$, $\mathcal{T}_{\nu,\theta \eta}$ and $T_{\theta \eta}$  by
\begin{align}\label{TMF}
\begin{split}
&\rho_{\eta}=\int_{\mathbb{R}^3\times \mathbb{R}_+}F_{\eta}dvdI,\quad
\rho_{\eta} U_{\eta}=\int_{\mathbb{R}^3\times \mathbb{R}_+}F_{\eta}vdvdI,\cr
&\rho_{\eta}\mathcal{T}_{\nu,\theta\eta}+\frac{\theta}{3+\delta}\rho_{\eta}|U_{\eta}|^2Id
+(1-\theta)\left\{\frac{1-\nu}{3}\rho_{\eta}|U_{\eta}|^2Id+\nu\rho_{\eta} U_{\eta}\otimes U_{\eta}\right\}\cr
&\qquad=\theta\left\{\int_{\mathbb{R}^3\times \mathbb{R}_+}F_{\eta}\left(\frac{1}{3+\delta}|v|^2+\frac{2}{3+\delta}I^{2/\delta}\right)
dvdI\right\}Id\cr
&\qquad+(1-\theta)\left\{\int_{\mathbb{R}^3\times \mathbb{R}_+}F_{\eta}\left(\frac{1-\nu}{3}|v|^2Id+\nu v\otimes v\right)dvdI\right\},\cr
&\rho_{\eta} T_{\theta\eta}+\frac{\theta}{3+\delta}\rho_{\eta}|U_{\eta}|^2\cr
&\qquad=\theta\left\{\int_{\mathbb{R}^3\times \mathbb{R}_+}F_{\eta}\left(\frac{1}{3+\delta}|v|^2+\frac{2}{3+\delta}I^{2/\delta}\right)dvdI\right\}\cr
&\qquad+(1-\theta)\left\{\frac{2}{\delta}\int_{\mathbb{R}^3\times \mathbb{R}_+}F_{\eta}I^{2/\delta}dvdI\right\},
\end{split}
\end{align}
and the transitional polyatomic Maxwellian:
\begin{eqnarray}\label{poly maxwellian eta}
\mathcal{M}_{\nu,\theta}(\eta)= \frac{\rho_{\eta} \Lambda_{\delta}}{\sqrt{\det (2\pi \mathcal{T}_{\theta\eta}}) \,T_{\theta\eta}^{\frac{\delta}{2}}}\exp\left(-\frac{1}{2}(v-U_{\eta})^{\top}\mathcal{T}^{-1}_{\nu,\theta\eta}
(v-U_{\eta})-\frac{I^{\frac{2}{\delta}}}{T_{\theta\eta}}\right).
\end{eqnarray}
For simplicity, we set
\begin{align*}
A(\eta)&=\rho_{\eta},\cr
B(\eta)&=\rho_{\eta} U_{\eta},\cr
C(\eta)&=\rho_{\eta}\mathcal{T}_{\nu,\theta\eta}+\frac{\theta}{3+\delta}\rho_{\eta}|U_{\eta}|^2Id
+(1-\theta)\left\{\frac{1-\nu}{3}\rho_{\eta}|U_{\eta}|^2Id+\nu\rho_{\eta} U_{\eta}\otimes U_{\eta}\right\},\cr
D(\eta)&=\rho_{\eta} T_{\theta\eta}+\frac{\theta}{3+\delta}\rho_{\eta}|U_{\eta}|^2.
\end{align*}

%
%
%
Note that $A(0)=1$, $B(0)=0$, $C(0)=Id$, $D(0)=1$ since $F^0=m$, and the macroscopic fields can be recovered from the following relations:
\begin{align}\label{sense}
\begin{split}
\rho_{\eta}&=A(\eta),\cr
U_{\eta}&=\frac{B(\eta)}{A(\eta)},\cr
\mathcal{T}_{\nu,\theta\eta}&=\frac{A(\eta)C(\eta)-\left\{\frac{\theta}{3+\delta}|B(\eta)|^2Id+(1-\theta)\left(\frac{1-\nu}{3}|B(\eta)|^2+\nu B(\eta)\otimes B(\eta)\right)\right\}}{|A(\eta)|^2},\cr
T_{\theta\eta}&=\frac{A(\eta)D(\eta)-\frac{\theta}{3+\delta}|B(\eta)|^2}{|A(\eta)|^2}.
\end{split}
\end{align}

%
%
%
%
%
The following identity plays an important role throughout the linearization procedure.
\begin{lemma}\label{matrix} The Jacobian matrix $J(\eta)=\frac{\partial (\rho_{\eta}, U_{\eta}, \mathcal{T}_{\nu,\theta\eta},T_{\theta\eta})}{\partial\big(A(\eta),B(\eta),C(\eta),D(\eta)\big)}$ is given by
\begin{eqnarray*}
J(\eta)=
\left(\begin{array}{cccccccccccc}
1&0&0&0&0&0&0&0&0&0&0\\
-\frac{U_{\eta1}}{\rho_{\eta}}&\frac{1}{\rho_{\eta}}&0&0&0&0&0&0&0&0&0\\
-\frac{U_{\eta2}}{\rho}&0&\frac{1}{\rho_{\eta}}&0&0&0&0&0&0&0&0\\
-\frac{U_{\eta3}}{\rho_{\eta}}&0&0&\frac{1}{\rho_{\eta}}&0&0&0&0&0&0&0\\
\Lambda^{11}_{\eta}&J_+\frac{U_{\eta1}}{\rho_{\eta}}&J_-\frac{U_{\eta2}}{\rho_{\eta}}&J_-\frac{U_{\eta3}}{\rho_{\eta}}&\frac{1}{\rho_{\eta}}&0&0&0&0&0&0\\
\Lambda^{22}_{\eta}&J_-\frac{U_{\eta1}}{\rho_{\eta}}&J_+\frac{U_{\eta2}}{\rho}&J_-\frac{U_{\eta3}}{\rho_{\eta}}&0&\frac{1}{\rho_{\eta}}&0&0&0&0&0\\
\Lambda^{33}_{\eta}&J_-\frac{U_{\eta1}}{\rho_{\eta}}&J_-\frac{U_{\eta2}}{\rho}&J_+\frac{U_{\eta3}}{\rho_{\eta}}&0&0&\frac{1}{\rho_{\eta}}&0&0&0&0\\
\Lambda^{12}_{\eta}&-\nu\frac{ U_{\eta2}}{\rho_{\eta}}&-\nu\frac{U_{\eta1}}{\rho_{\eta}}&0&0&0&0&\frac{1}{\rho_{\eta}}&0&0&0\\
\Lambda^{23}_{\eta}&0&-\nu\frac{U_{\eta3}}{\rho_{\eta}}&-\nu\frac{U_{\eta2}}{\rho_{\eta}}&0&0&0&0&\frac{1}{\rho_{\eta}}&0&0\\
\Lambda^{31}_{\eta}&-\nu\frac{U_{\eta3}}{\rho_{\eta}}&0&-\nu\frac{U_{\eta1}}{\rho_{\eta}}&0&0&0&0&0&\frac{1}{\rho_{\eta}}&0\\
\Omega_{\eta}&-\frac{2\theta}{3+\delta}\frac{U_{\eta1}}{\rho_{\eta}}&-\frac{2\theta}{3+\delta}\frac{U_{\eta2}}{\rho_{\eta}}&-\frac{2\theta}{3+\delta}\frac{U_{3\eta}}{\rho_{\eta}}&0&0&0&0&0&0&\frac{1}{\rho_{\eta}}
\end{array}
\right),
\end{eqnarray*}
where $\Lambda^{ij}_{\eta}$, $\Omega_{\eta}$ and $J_{\pm}$  are
\begin{align*}
\Lambda^{ii}_{\eta}&=\frac{1}{\rho_{\eta}}\left\{-\mathcal{T}^{ii}_{\nu,\theta\eta}
+\left(\frac{\theta}{3+\delta}+(1-\theta)\frac{1-\nu}{3}\right)|U_{\eta}|^2+\nu(1-\theta) U^2_{\eta i}\right\},\cr
\Lambda^{ij}_{\eta}&=\frac{1}{\rho_{\eta}}\left\{-\mathcal{T}^{ij}_{\nu,\theta\eta}+\nu (1-\theta)U_{\eta i}U_{\eta j}\right\},\cr
\Omega_{\eta}&=\frac{1}{\rho}\left(-T_{\theta\eta}+\frac{\theta}{3+\delta}|U_{\eta}|^2\right),\cr
J_+&=-\left\{\frac{\theta}{3+\delta}+(1-\theta)\frac{1+2\nu}{3}\right\},\cr
J_-&=-\left\{\frac{\theta}{3+\delta}+(1-\theta)\frac{1-\nu}{3}\right\}.
\end{align*}
\end{lemma}
\begin{proof}
It follows from a straightforward computation using the relations (\ref{sense}). We omit it.
\end{proof}
The following corollary comes immediately.
\begin{corollary}\label{immediate}
When $F^0=m$, the Jacobian is given by
\begin{eqnarray*}
J(0)=
\left(\begin{array}{ccccccccccc}
1&0&0&0&0&0&0&0&0&0&0\\
0&1&0&0&0&0&0&0&0&0&0\\
0&0&1&0&0&0&0&0&0&0&0\\
0&0&0&1&0&0&0&0&0&0&0\\
-1&0&0&0&1&0&0&0&0&0&0\\
-1&0&0&0&0&1&0&0&0&0&0\\
-1&0&0&0&0&0&1&0&0&0&0\\
0&0&0&0&0&0&0&1&0&0&0\\
0&0&0&0&0&0&0&0&1&0&0\\
0&0&0&0&0&0&0&0&0&1&0\\
-1&0&0&0&0&0&0&0&0&0&1
\end{array}
\right).
\end{eqnarray*}
\end{corollary}

\begin{lemma} \label{22}We have
\begin{align*}
&(1)~ \frac{\partial\mathcal{M}_{\nu,\theta}(\eta)}{\partial \rho_{\eta}}=\frac{1}{\rho_{\eta}}\mathcal{M}_{\nu,\theta}(\eta),\cr
&(2)~ \nabla_{U_{\eta}}\mathcal{M}_{\nu,\theta}(\eta)
=\frac{1}{2}\left\{\mathcal{T}^{-1}_{\nu,\theta\eta}(v-U_{\eta})+(v-U_{\eta})^{\top}\mathcal{T}^{-1}_{\nu,\theta\eta}\right\}
\mathcal{M}_{\nu,\theta}(\eta),\cr
&(3)~ \frac{\partial\mathcal{M}_{\nu,\theta}(\eta)}{\partial \mathcal{T}^{ii}_{\nu,\theta\eta}}
=\frac{1}{2}\left\{-\frac{1}{\det\mathcal{T}_{\nu,\theta\eta}}
\frac{\partial\big(\det\mathcal{T}_{\nu,\theta\eta}\big)}{\partial \mathcal{T}_{\nu,\theta\eta}^{ij}}
+\big\{(v-U_{\eta})^{\top}\mathcal{T}^{-1}_{\nu,\theta\eta}\}_{i}^2\right\}\mathcal{M}_{\nu,\theta}(\eta),\cr
&(4)~ \frac{\partial\mathcal{M}_{\nu,\theta}(\eta)}{\partial \mathcal{T}^{ij}_{\nu,\theta\eta}}
=\frac{1}{2}\left\{-\frac{1}{\det\mathcal{T}_{\nu,\theta\eta}}
\frac{\partial\big(\det\mathcal{T}_{\nu,\theta\eta}\big)}{\partial \mathcal{T}_{\nu,\theta\eta}^{ij}}
+2\big\{(v-U_{\eta})^{\top}\mathcal{T}^{-1}_{\nu,\theta\eta}\}_{i}
\big\{\mathcal{T}^{-1}_{\nu,\theta\eta}(v-U_{\eta})\big\}_j\right\}\mathcal{M}_{\nu,\theta}(\eta),\cr
&(5)~ \frac{\partial\mathcal{M}_{\nu,\theta}(\eta)}{\partial T_{\theta\eta}}
=\left\{\frac{2I^{2/\delta}-\delta T_{\theta\eta}}{2T_{\theta\eta}^2}\right\}\mathcal{M}_{\nu,\theta}(\eta).
\end{align*}
\end{lemma}
\begin{proof}
(1), (2) and (5) follow from direct computations. The proofs for (3) and (4) are similar. We only prove $(4)$. We first compute
\begin{align*}
\frac{\partial\mathcal{M}_{\nu,\theta}(\eta)}{\partial \mathcal{T}^{ij}_{\nu,\theta\eta}}=\frac{1}{2}\left\{-\frac{1}{\det(\mathcal{T}_{\nu,\theta\eta})}
\frac{\partial\big(\det\mathcal{T}_{\nu,\theta\eta}\big)}{\partial \mathcal{T}_{\nu,\theta\eta}^{ij}}
-(v-U_{\eta})^{\top}\left(\frac{\partial \mathcal{T}_{\nu,\theta\eta}^{-1}}{\partial \mathcal{T}^{ij}_{\nu,\theta\eta}}\right)
(v-U_{\eta})\right\}\mathcal{M}_{\nu,\theta}(\eta).
\end{align*}
We then observe that for  any invertible matrix $A$
\begin{eqnarray}\label{gives}
\partial \big\{A^{-1}\big\}=-A^{-1}\left\{\partial A\right\}A^{-1},
\end{eqnarray}
which is obtained by applying $\partial$ on both sides of $AA^{-1}=I$:
\[
\big\{\partial A\big\}A^{-1}+A\partial\big\{A^{-1}\big\}=\partial I=0.
\]
Therefore,
\begin{align*}
(v-U_{\eta})^{\top}\left(\frac{\partial \mathcal{T}_{\nu,\theta\eta}^{-1}}{\partial \mathcal{T}^{ij}_{\nu,\theta\eta}}\right)
(v-U_{\eta})=(v-U_{\eta})^{\top}\mathcal{T}^{-1}_{\nu,\theta\eta}\left(\frac{\partial \mathcal{T}_{\nu,\theta\eta}}{\partial \mathcal{T}^{ij}_{\nu,\theta\eta}}\right)\mathcal{T}^{-1}_{\nu,\theta\eta}
(v-U_{\eta}).
\end{align*}
Finally, since $\left(\frac{\partial \mathcal{T}_{\nu,\theta\eta}}{\partial \mathcal{T}^{ij}_{\nu,\theta\eta}}\right)$
is a matrix whose only non-zero element is $ij$th and $ji$th elements, this simplifies further
\begin{align*}
(v-U_{\eta})^{\top}\mathcal{T}^{-1}_{\nu,\theta\eta}\left(\frac{\partial \mathcal{T}_{\nu,\theta\eta}}{\partial \mathcal{T}^{ij}_{\nu,\theta\eta}}\right)\mathcal{T}^{-1}_{\nu,\theta\eta}
(v-U_{\eta})=2\big\{(v-U_{\eta})^{\top}\mathcal{T}^{-1}_{\nu,\theta\eta}\}_{i}
\big\{\mathcal{T}^{-1}_{\nu,\theta\eta}(v-U_{\eta})\big\}_j.
\end{align*}
This completes the proof.
\end{proof}
\begin{corollary}\label{xx} When $\eta=0$, we have
\begin{align*}
&(1)~ \frac{\partial\mathcal{M}_{\nu,\theta}(0)}{\partial \rho_{\eta}}=m,\cr
&(2)~ \frac{\partial \mathcal{M}_{\nu,\theta}(0)}{\partial U_{i\eta}}=v_im,\hspace{1.6cm} (i=1,2,3)\cr
&(3)~ \frac{\partial\mathcal{M}_{\nu,\theta}(0)}{\partial \mathcal{T}^{ii}_{\nu,\theta\eta}}
=\frac{v_i^2-1}{2}m, \hspace{0.85cm}(1\leq i=j\leq 3)\cr
&(4)~ \frac{\partial\mathcal{M}_{\nu,\theta}(0)}{\partial \mathcal{T}^{ij}_{\nu,\theta\eta}}
=v_iv_jm,\hspace{1.2cm} (1\leq i\neq j\leq3)\cr
&(5)~ \frac{\partial\mathcal{M}_{\nu,\theta}(0)}{\partial T_{\theta\eta}}
=\left\{\frac{2I^{2/\delta}-\delta }{2}\right\}m.
\end{align*}
\end{corollary}
\begin{proof}
Note that when $\eta=0$, $F_{\eta}$ reduces to $m$. Therefore, the result follows by inserting
$\rho_0=1$, $U_0=0$, $\mathcal{T}_0=Id$, $T_0=1$ to Lemma \ref{22}.
\end{proof}

%
\subsection{Linearized relaxation operator}
We consider the transitional polyatomic Maxwellian as a function of $\eta$ and set
\begin{align*}
g(\eta)=\mathcal{M}_{\nu,\theta}\left(A(\eta),B(\eta), C(\eta),D(\eta)\right).
\end{align*}
Here, we view $C$ as a 6 dimensional vector $\big(C_{11}, C_{22}, C_{33}, C_{12}, C_{23}, C_{31}\big)$ by symmetry.
Note that $g(\eta)$ depicts the transition from the polyatomic local Maxwellian $\mathcal{M}_{\nu,\theta}(F)$
to the polyatomic global Maxwellian $m(v,I)$.
We expand it using the Taylor's theorem:
\begin{eqnarray}\label{Taylor Expansion}
g(1)=g(0)+g^{\prime}(0)+\int^1_0g^{\prime\prime}(\eta)(1-\eta)d\eta.
\end{eqnarray}
Clearly,
\[
g(0)=m, \mbox{ and } g(1)=\mathcal{M}_{\nu,\theta}(F).
\]
The calculation of the second and third term in the right hand side of (\ref{Taylor Expansion}) is carried out in the following theorem and Proposition \ref{prop21} respectively.
\begin{theorem}\label{expansion}
$g^{\prime}(0)$ is given by
\begin{equation*}
g^{\prime}(0)=(P_{\nu,\theta}f)\sqrt{m},
\end{equation*}
where  $P_{\nu,\theta}$ is defined by
\begin{eqnarray*}
P_{\nu,\theta}f\equiv \theta P_{p}f+(1-\theta)\big\{ P_mf+ \nu(P_1f+P_2f)\big\}.
\end{eqnarray*}
\noindent(1) $P_p$: polyatomic projection:
\begin{align*}
P_{p}f&=\left(\int_{\mathbb{R}^3\times \mathbb{R}_+} f\sqrt{m}dvdI\right) \sqrt{m} \cr
&+\left(\int_{\mathbb{R}^3\times \mathbb{R}_+} fv\sqrt{m}dvdI\right)\cdot v\sqrt{m}\cr
&+\left\{\int_{\mathbb{R}^3\times \mathbb{R}_+} f\left(\frac{(|v|^2-3)+(2I^{\frac{2}{\delta}}-\delta)}{\sqrt{2(3+\delta)}}\right)\sqrt{m}\,dvdI\right\}
\left(\frac{(|v|^2-3)+(2I^{2/\delta}-\delta)}{\sqrt{2(3+\delta)}}\right)\sqrt{m}\,,
\end{align*}
(2) $P_m$: monatomic-like projection:
\begin{align*}
P_{m}f&=\left(\int_{\mathbb{R}^3\times \mathbb{R}_+} f\sqrt{m}dvdI\right)\sqrt{m}\cr
&+\left(\int_{\mathbb{R}^3\times \mathbb{R}_+} fv\sqrt{m}dvdI\right)\cdot v\sqrt{m}\cr
&+\left\{\int_{\mathbb{R}^3\times \mathbb{R}_+} f\left(\frac{|v|^2-3}{\sqrt{6}}\right)\sqrt{m}dvdI\right\}
\left(\frac{|v|^2-3}{\sqrt{6}}\right)\sqrt{m}\cr
&+\left\{\int_{\mathbb{R}^3\times \mathbb{R}_+} f\left(\frac{2I^{\frac{2}{\delta}}-\delta}{\sqrt{2\delta}}\right)\sqrt{m}dvdI\right\}
\left(\frac{2I^{\frac{2}{\delta}}-\delta}{\sqrt{2\delta}}\right)\sqrt{m},
\end{align*}
(3) $P_1$ $\&$ $P_2$: non-diagonal projections:
\begin{align*}
P_{1}f&=\sum_{i<j}
\left\{\int_{\mathbb{R}^3\times \mathbb{R}_+} f\Big(\frac{3v_i^2-|v|^2}{3\sqrt{2}}\Big)\sqrt{m}dvdI\right\}\left(\frac{3v_1^2-|v|^2}{3\sqrt{2}}\right)\sqrt{m}\,,\cr
P_{2}f&= \sum_{i< j}\left(\int_{\mathbb{R}^3\times \mathbb{R}_+} f v_iv_j\sqrt{m}dvdI\right)
v_iv_j\sqrt{m}\,.
\end{align*}
\end{theorem}
\begin{proof}
By chain rule,
\begin{align}\label{g0}
\begin{split}
g^{\prime}(0)&=
A^{\prime}(0)\frac{\partial \mathcal{M}_{\nu,\theta}(0)}{\partial A}
+B^{\prime}(0)\frac{\partial \mathcal{M}_{\nu,\theta}(0)}{\partial B}
+C^{\prime}(0)\frac{\partial \mathcal{M}_{\nu,\theta}(0)}{\partial C}
+D^{\prime}(0)\frac{\partial \mathcal{M}_{\nu,\theta}(0)}{\partial D}\cr
&=\nabla_{(A,B,C,D)}\mathcal{M}_{\nu,\theta}(0)\cdot\big(A^{\prime}(0),B^{\prime}(0),C^{\prime}(0),D^{\prime}(0)\big)\cr
&=\left\{\nabla_{(\rho_{\eta},U_{\eta},\mathcal{T}_{\nu,\theta\eta},T_{\theta\eta})}\mathcal{M}_{\nu,\theta}(0)
J(0)\right\}
\big(A^{\prime}(0),B^{\prime}(0),C^{\prime}(0),D^{\prime}(0)\big)^{\top},
\end{split}
\end{align}
where $J(\eta)$ denotes the Jacobian matrix between the translational macroscopic fields given in Lemma \ref{matrix}.
Recalling Corollary \ref{immediate}  and Corollary \ref{xx}, we see that
\begin{align*}
&\nabla_{(\rho_{\eta},U_{\eta},\mathcal{T}_{\nu,\theta \eta},T_{\theta\eta})}\mathcal{M}_{\nu,\theta}(0)J(0)\cr
&\hspace{0.2cm}=\left(\,1,v_1,v_2,v_3,\frac{v_1^2-1}{2},\frac{v_2^2-1}{2},\frac{v_3^2-1}{2},
v_1v_2,v_2v_3,v_3v_1,\frac{2I^{\frac{2}{\delta}}-\delta}{2}\,\right)m\,J(0)\cr
&\hspace{0.2cm}=\left(1-\frac{|v|^2-3}{2}-\frac{2I^{\frac{2}{\delta}}-\delta}{2}
,v_1,v_2,v_3,\frac{v_1^2-1}{2},\frac{v_2^2-1}{2},\frac{v_3^2-1}{2},
v_1v_2,v_2v_3,v_3v_1,\frac{2I^{\frac{2}{\delta}}-\delta}{2}\right)m.
\end{align*}
On the other hand, we find
\begin{align*}
\left(
\begin{array}{c}
A^{\prime}(0)\cr
B^{\prime}(0)\cr
C^{\prime}(0)\cr
D^{\prime}(0)
\end{array}
\right)
=\left(
\begin{array}{c}
\int f\sqrt{m}dvdI\\
\int fv_1\sqrt{m}dvdI\\
\int fv_2\sqrt{m}dvdI\\
\int fv_3\sqrt{m}dvdI\\
\theta\int f\left(\frac{1}{3+\delta}|v|^2+\frac{2}{3+\delta}I^{\frac{2}{\delta}}\right)\sqrt{m}dvdI
+(1-\theta)\left\{\int f\left(\frac{1-\nu}{3}|v|^2+\nu v_1^2\right)\sqrt{m}dvdI\right\}\\
\theta\int f\left(\frac{1}{3+\delta}|v|^2+\frac{2}{3+\delta}I^{\frac{2}{\delta}}\right)
\sqrt{m}dvdI
+(1-\theta)\left\{\int f\left(\frac{1-\nu}{3}|v|^2+\nu v_1^2\right)\sqrt{m}dvdI\right\}\\
\theta\int f\left(\frac{1}{3+\delta}|v|^2+\frac{2}{3+\delta}I^{\frac{2}{\delta}}\right)\sqrt{m}
dvdI
+(1-\theta)\left\{\int f\left(\frac{1-\nu}{3}|v|^2+\nu v_1^2\right)\sqrt{m}dvdI\right\}\\
(1-\theta)\nu\int f v_1v_2\sqrt{m}dvdI\\
(1-\theta)\nu\int f v_2v_3\sqrt{m}dvdI\\
(1-\theta) \nu\int f v_3v_1\sqrt{m}dvdI\\
\theta\int f\left(\frac{1}{3+\delta}|v|^2+\frac{2}{3+\delta}I^{\frac{2}{\delta}}\right)\sqrt{m}dvdI
+(1-\theta)\left(\frac{2}{\delta}\int fI^{\frac{2}{\delta}}\sqrt{m}dvdI\right)
\end{array}\right)
\end{align*}
Inserting these identities into (\ref{g0})
\begin{align*}
g^{\prime}(0)m^{-1}
&=\left(\int f\sqrt{m}dvdI\right)\Big(1-\frac{|v|^2-3}{2}-\frac{2I^{\frac{2}{\delta}}-\delta}{2}\Big)\cr
&+\left(\int fv\sqrt{m}dvdI\right)\cdot v\cr
&+\sum_{i=1}^3\left[\theta\int f\left(\frac{1}{3+\delta}|v|^2+\frac{2}{3+\delta}I^{\frac{2}{\delta}}\right)\sqrt{m}dvdI\right.\cr
&\hspace{1.2cm}+\left.(1-\theta)\left\{\int f\left(\frac{1-\nu}{3}|v|^2+\nu v_i^2\right)\sqrt{m}dvdI\right\}\right]\left(\frac{v_i^2-1}{2}\right)\cr
&+(1-\theta) \nu\sum_{i<j}\left(\int f v_iv_j\sqrt{m}dvdI\right)v_iv_j\cr
&+\left[\theta\int f\left(\frac{1}{3+\delta}|v|^2+\frac{2}{3+\delta}I^{\frac{2}{\delta}}\right)\sqrt{m}dvdI
+(1-\theta)\left(\frac{2}{\delta}\int fI^{\frac{2}{\delta}}\sqrt{m}dvdI\right)\right]\left(\frac{2I^{\frac{2}{\delta}}-\delta}{2}\right)\cr
&\equiv I_1+I_2+I_3+I_4+I_5.
\end{align*}
For later computation, we further decompose $I_1$ and $I_5$ as follows:
\begin{align*}
&I^1_1=\int f\sqrt{m}dvdI,\,
I^2_1=-\left(\int f\sqrt{m}dvdI\right)\Big(\frac{|v|^2-3}{2}\Big),\cr
&\hspace{1cm}I^3_1=-\left(\int f\sqrt{m}dvdI\right)\Big(\frac{2I^{\frac{2}{\delta}}-\delta}{2}\Big),
\end{align*}
and
\begin{align*}
I^1_5&=\theta\left\{\int f\left(\frac{1}{3+\delta}|v|^2+\frac{2}{3+\delta}I^{\frac{2}{\delta}}\right)\sqrt{m}dvdI
\right\}\Big(\frac{2I^{\frac{2}{\delta}}-\delta}{2}\Big)\cr
I^2_5&=(1-\theta)\left(\frac{2}{\delta}\int fI^{\frac{2}{\delta}}\sqrt{m}dvdI\right)\Big(\frac{2I^{\frac{2}{\delta}}-\delta}{2}\Big).
\end{align*}
We now rearrange these terms so that (1) the polyatomic part and monatomic-like part are separated, and (2) the orthogonality between the components are clearly revealed, as is given in the statement of the Lemma \ref{Nilpotent} later. For this, we need some preliminary calculations:\newline

\noindent {\bf Step I: $I_3=A_1+A_2+A_3$}\newline First we compute the summation in $I_3$ to obtain
\begin{align}\label{I3 rewrite}
\begin{split}
I_3
&=\theta\left\{\int f\left(\frac{|v|^2+2I^{\frac{2}{\delta}}}{3+\delta}\right)\sqrt{m}dvdI\right\}\frac{|v|^2-3}{2}\cr
&+(1-\theta)\frac{1-\nu}{3}\left(\int f|v|^2\sqrt{m}dvdI\right)\frac{|v|^2-3}{2}\cr
&+(1-\theta)\nu\sum_{1\leq i\leq 3}\left(\int f\,v^2_i\sqrt{m}dvdI\right)\frac{v_i^2-1}{2}.
\end{split}
\end{align}
The last term  can be decomposed as
\begin{align}\label{simplified}
\begin{split}
\sum_{1\leq i\leq 3}\left(\int f\,v^2_idvdI\right)\frac{v_i^2-1}{2}
&=\sum_{1\leq i\leq 3}\left(\int f\left(\frac{3v^2_i-|v|^2}{3}\right)\sqrt{m}dvdI\right)\frac{3v_i^2-|v|^2}{6}\cr
&+\sum_{1\leq i\leq 3}\left\{\int f\left(\frac{3v_i^2-|v|^2}{3}\right)\sqrt{m}dvdI\right\}\frac{|v|^2-3}{6}\cr
&+\sum_{1\leq i\leq 3}\left\{\int f\left(\frac{|v|^2}{3}\right)\sqrt{m}dvdI\right\}\frac{3v^2_i-|v|^2}{6}\cr
&+\sum_{1\leq i\leq 3}\left\{\int f\left(\frac{|v|^2}{3}\right)\sqrt{m}dvdI\right\}\frac{|v|^2-3}{6}.
\end{split}
\end{align}
The second and third terms vanish due to
\[
\sum_{1\leq i\leq 3}\left\{3v^2_i-|v|^2\right\}=0,
\]
and the last term is
\begin{align*}
\left\{\int f|v|^2\sqrt{m}dvdI\right\}\frac{|v|^2-3}{6}.
\end{align*}
so that (\ref{simplified}) is reduced to
\begin{align*}
\sum_{1\leq i\leq 3}\left(\int fv^2_i\sqrt{m}dvdI\right)\frac{v_i^2-1}{2}
&=\sum_{1\leq i\leq 3}\left\{\int f\left(\frac{3v^2_i-|v|^2}{3}\right)\sqrt{m}dvdI\right\}\frac{3v_i^2-|v|^2}{6}\cr
&+\left(\int f|v|^2\sqrt{m}dvdI\right)\frac{|v|^2-3}{6}.
\end{align*}
We plug this into (\ref{I3 rewrite}),
\begin{align*}
I_3&=\theta\left\{\int f\left(\frac{|v|^2+2I^{\frac{2}{\delta}}}{3+\delta}\right)\sqrt{m}dvdI\right\}\frac{|v|^2-3}{2}\cr
&+(1-\theta)\frac{1-\nu}{3}\left(\int f \,|v|^2dvdI\right)\frac{|v|^2-3}{2}\cr
&+(1-\theta)\nu\sum_{1\leq i\leq 3}
\left\{\int f\Big(\frac{3v_i^2-|v|^2}{3}\Big)\sqrt{m}dvdI\right\}\frac{3v_1^2-|v|^2}{6}\cr
&+(1-\theta)\nu\left(\int f \,|v|^2dvdI\right)\frac{|v|^2-3}{6}.
\end{align*}
and put together the second and fourth terms to get
\begin{align*}
I_3
&=\theta\left\{\int f\left(\frac{|v|^2+2I^{\frac{2}{\delta}}}{3+\delta}\right)\sqrt{m}dvdI\right\}\frac{|v|^2-3}{2}\cr
&+(1-\theta)\left(\int f\frac{|v|^2}{3}dvdI\right)\frac{|v|^2-3}{2}\cr
&+\nu(1-\theta)\sum_{1\leq i\leq 3}
\left\{\int f\Big(\frac{3v_i^2-|v|^2}{3}\Big)\sqrt{m}dvdI\right\}\frac{3v_i^2-|v|^2}{6}\cr
&\equiv A^1+A^2+A^3.
\end{align*}
{\bf Step II :}  $\theta(I^2_1+I^3_1)+A^1+I^1_5$ \newline
 We combine the first term of $I_5$ with $A^1$:
\begin{align*}
A^1+I^1_5=
\theta\left\{\int f\left(\frac{|v|^2+2I^{\frac{2}{\delta}}}{3+\delta}\right)\sqrt{m}\,dvdI\right\}
\left(\frac{(|v|^2-3)+(2I^{2/\delta}-\delta)}{2}\right),
\end{align*}
Therefore, adding $\theta$ portion of the second, third term of $I_1$ to $I_5+A^1$, we obtain
\begin{align*}
&\theta\left(I^2_1+I^{3}_1\right)+A^1+I^{1}_5\cr
&\qquad=-\theta\left\{\int f\sqrt{m}\,dvdI\right\}
\left(\frac{(|v|^2-3)+(2I^{2/\delta}-\delta)}{2}\right)\cr
&\qquad+\theta\left\{\int f\left(\frac{|v|^2+2I^{\frac{2}{\delta}}}{3+\delta}\right)\sqrt{m}\,dvdI\right\}
\left(\frac{(|v|^2-3)+(2I^{2/\delta}-\delta)}{2}\right)\cr
&\qquad=\theta\left\{\int f\left(\frac{(|v|^2-3)+(2I^{\frac{2}{\delta}}-\delta)}{3+\delta}\right)\sqrt{m}\,dvdI\right\}
\left(\frac{(|v|^2-3)+(2I^{2/\delta}-\delta)}{2}\right).\cr
\end{align*}
{\bf Step III:} Now,  we rewrite $I_1+\cdots +I_5$ as
\begin{align*}
&I_1+I_2+I_3+I_4+I_5\cr
&\qquad=\big(I^1_1+I^2_1+I^3_1\big)+I_2+\big(A_1+A_2+A_3\big)+I_4+(I^1_5+I^2_5)\cr
&\qquad=I^1_1+I_2+\left\{\theta\big(I^2_1+I^3_1\big)+A_1+I^1_5\right\}+A_2+A_3+(1-\theta)\big(I^2_1+I^3_1\big)+I_4+I^2_5
\end{align*}
and insert the above computations in step I and step II, to derive
\begin{align*}
&I_1+\cdots+I_5\cr
&\qquad=\left(\int f\sqrt{m}dvdI\right) \cr
&\qquad+\left(\int fv\sqrt{m}dvdI\right)\cdot v\cr
&\qquad+\theta\left\{\int f\left(\frac{(|v|^2-3)+(2I^{\frac{2}{\delta}}-\delta)}{3+\delta}\right)\sqrt{m}\,dvdI\right\}
\left(\frac{(|v|^2-3)+(2I^{2/\delta}-\delta)}{2}\right)\cr
&\qquad+(1-\theta)\left(\int f\frac{|v|^2}{3}\sqrt{m}dvdI\right)\left(\frac{|v|^2-3}{2}\right)\cr
&\qquad+\nu(1-\theta)\sum_{1\leq i\leq 3}
\left\{\int f\Big(\frac{3v_i^2-|v|^2}{3}\Big)\sqrt{m}dvdI\right\}\frac{3v_1^2-|v|^2}{6}\cr
&\qquad-(1-\theta)\left\{\int f\sqrt{m}dvdI\right\} \left(\frac{|v|^2-3}{2}\right)\cr
&\qquad-(1-\theta)\left\{\int f\sqrt{m}dvdI\right\} \left(\frac{2I^{\frac{2}{\delta}}-\delta}{2}\right)\cr
&\qquad+(1-\theta) \nu\sum_{i< j}\left(\int f v_iv_j\sqrt{m} dvdI\right)v_iv_j\cr
&\qquad+(1-\theta)\left\{\frac{2}{\delta}\left(\int fI^{\frac{2}{\delta}} dvdI\right)\right\}\left(\frac{2I^{\frac{2}{\delta}}-\delta}{2}\right).
\end{align*}
Note that the 4th and the 6th terms on the r.h.s put together give
\begin{align*}
(1-\theta)\left(\int f\frac{|v|^2-3}{\sqrt{6}}dvdI\right)\frac{|v|^2-3}{\sqrt{6}}.
\end{align*}
Likewise, the 7th and the 9th term on the r.h.s can be combined to yield
\begin{align*}
(1-\theta)\left\{\int f\left(\frac{2I^{\frac{2}{\delta}}-\delta}{\sqrt{2\delta}}\right)\sqrt{m}dvdI\right\} \left(\frac{2I^{\frac{2}{\delta}}-\delta}{\sqrt{2\delta}}\right).
\end{align*}
In conclusion,
\begin{align*}
I_1+\cdots+I_5&=\left(\int f\sqrt{m}dvdI\right)  \cr
&+\left(\int fv\sqrt{m}dvdI\right)\cdot v\cr
&+\theta\left\{\int f\left(\frac{(|v|^2-3)+(2I^{\frac{2}{\delta}}-\delta)}{\sqrt{2(3+\delta)}}\right)\sqrt{m}\,dvdI\right\}
\left(\frac{(|v|^2-3)+(2I^{2/\delta}-\delta)}{\sqrt{2(3+\delta)}}\right)\cr
&+(1-\theta)\left\{\int f\left(\frac{|v|^2-3}{\sqrt{6}}\right)\sqrt{m}dvdI\right\}
\left(\frac{|v|^2-3}{\sqrt{6}}\right)\cr
&+(1-\theta)\left\{\int f\left(\frac{2I^{\frac{2}{\delta}}-\delta}{\sqrt{2\delta}}\right)\sqrt{m}dvdI\right\}
\left(\frac{2I^{\frac{2}{\delta}}-\delta}{\sqrt{2\delta}}\right)\cr
&+\nu(1-\theta)\sum_{i}
\left\{\int f\Big(\frac{3v_i^2-|v|^2}{3\sqrt{2}}\Big)\sqrt{m}dvdI\right\}\left(\frac{3v_i^2-|v|^2}{3\sqrt{2}}\right)\cr
&+\nu(1-\theta) \sum_{i< j}\left\{\int f v_iv_j\sqrt{m}dvdI\right\}
v_iv_j.
\end{align*}
Finally, we split the first two term as
\begin{align*}
&\left(\int f\sqrt{m}dvdI\right)  +\left(\int fv\sqrt{m}dvdI\right)\cdot v\cr
&\hspace{1.5cm}=
 \theta\left\{\left(\int f\sqrt{m}dvdI\right)  +\left(\int fv\sqrt{m}dvdI\right)\cdot v\right\}\cr
&\hspace{1.5cm}+(1-\theta)\left\{\left(\int f\sqrt{m}dvdI\right)  +\left(\int fv\sqrt{m}dvdI\right)\cdot v\right\},
\end{align*}
and gather terms with $\theta$ and $(1-\theta)$ separately, which
are $P_{p}f$ and $P_mf+\nu(P_1+P_2)f$ respectively.
\end{proof}

We now move on to the nonlinear term.
In the following, the polynomials $P^{\mathcal{M}}_{ij}$, $R^{\mathcal{M}}_{ij}$ are generically defined in the sense that their exact form may vary line after line
but can be explicitly computed in principle. Note that explicit form is not relevant as long as they satisfy the structural assumptions $\mathcal{H}_{\mathcal{M}}$ below.
%
%
%
%
%
\begin{proposition}\label{prop21}
$g^{\prime\prime}(\eta)$ is given by
\begin{eqnarray*}
g^{\prime\prime}(\eta)=\sum_{i,j}\bigg\{\int^1_0\frac{P^{\mathcal{M}}_{i,j}(\rho_{\eta},v-U_{\eta}, \mathcal{T}^{-1}_{\nu,\theta \eta}, I^{2/\delta}, T_{\theta\eta})}{R^{\mathcal{M}}_{i,j}(\rho_{\eta},\det\mathcal{T}_{\nu,\theta\eta},T_{\theta\eta})}
\mathcal{M}_{\nu,\theta}(\eta)(1-\eta)d\eta\bigg\}\langle f,e_i\rangle_{L^2_{v,I}}\langle f,e_j\rangle_{L^2_{v,I}},
\end{eqnarray*}
where $P^{\mathcal{M}}_{i,j}(x_1,\ldots,x_n)$ and $R^{\mathcal{M}}_{i,j}(x_1,\ldots, x_n)$ are generically defined polynomials satisfying the following structural assumptions $(\mathcal{H}_{\mathcal{M}})$:
\begin{itemize}
\item $(\mathcal{H}_{\mathcal{M}}1)$ $P^{\mathcal{M}}_{i,j}$ is a polynomial such that $P_{i,j}(0,0,\ldots, 0)=0.$
\item $(\mathcal{H}_{\mathcal{M}}2)$ $R^{\mathcal{M}}_{i,j}$ is a monomial.
\end{itemize}
and
\begin{align*}
e_1&=\sqrt{m},\\
e_{i+1}&=v_i\sqrt{m},\quad(i=1,2,3),\\
e_j&=\left\{\theta\left(\frac{1}{3+\delta}|v|^2+\frac{2}{3+\delta}I^{\frac{2}{\delta}}\right)
+(1-\theta)\left(\frac{1-\nu}{3}|v|^2+\nu v_i^2\right)\right\}\sqrt{m},\quad (j=5,6,7)\\
e_8&=\nu(1-\theta) v_1v_2\sqrt{m},\\
e_9&=\nu(1-\theta) v_2v_3\sqrt{m},\\
e_{10}&=\nu(1-\theta)  v_3v_1\sqrt{m},\\
e_{11}&=\left\{\theta\left(\frac{1}{3+\delta}|v|^2+\frac{2}{3+\delta}I^{\frac{2}{\delta}}\right)
+(1-\theta)\left(\frac{2}{\delta}I^{\frac{2}{\delta}}\right)\right\}\sqrt{m}.
\end{align*}
\end{proposition}
\begin{proof}
For a matrix $A$, let $A_k$ and $A_{k\ell}$ denote the $k$-th column of $A$ and the
$k\ell$ element respectively. For simplicity, we set
\begin{align*}
X_{\eta}&=\big(\rho_{\eta},U_{\eta},\mathcal{T}_{\nu,\theta\eta},T_{\theta\eta}\big),\cr
Y(\eta)&=\big(A(\eta), B(\eta), C(\eta), D(\eta)\big).
\end{align*}
Observe that each component of $Y(\eta)$ takes the form
\[
\int F_{\eta}P(v,I)dvdI=\int (m+\eta\sqrt{m}f)P(v,I)dvdI
\]
for some polynomial $P$. Therefore, $Y^{\prime}(\eta)$ does not depend on $\eta$, and we can write
\[
Y^{\prime}(\eta)=Y^{\prime}(0).
\]
Hence,  applying the chain rule, we compute
\begin{align*}
g^{\prime}(\eta)&=
A^{\prime}(0)\frac{\partial \mathcal{M}_{\nu,\theta}(\eta)}{\partial A}
+B^{\prime}(0)\frac{\partial \mathcal{M}_{\nu,\theta}(\eta)}{\partial B}
+C^{\prime}(0)\frac{\partial \mathcal{M}_{\nu,\theta}(\eta)}{\partial C}
+D^{\prime}(0)\frac{\partial \mathcal{M}_{\nu,\theta}(\eta)}{\partial D}\cr
&=\nabla_{(A,B,C,D)}\mathcal{M}_{\nu,\theta}(\eta)\cdot\big(A^{\prime}(0),B^{\prime}(0),C^{\prime}(0),D^{\prime}(0)\big)\cr
&=\nabla_{(\rho_{\eta},U_{\eta},\mathcal{T}_{\nu,\theta\eta},T_{\theta\eta})}\mathcal{M}_{\nu,\theta}(\eta)
J(\eta)Y^{\prime}(0)^{\top}\cr
&=\sum_i\left\{\nabla_{(\rho_{\eta},U_{\eta},\mathcal{T}_{\nu,\theta\eta},T_{\theta\eta})}\mathcal{M}_{\nu,\theta}(\eta)
\cdot J_i(\eta)\right\}Y^{\prime}_i(0),
\end{align*}
Taking the derivative again,
\begin{align*}
g^{\prime\prime}(\eta)&=\sum_i\left\{\Big(\nabla_{(\rho_{\eta},U_{\eta},\mathcal{T}_{\nu,\theta\eta},T_{\theta\eta})}\mathcal{M}_{\nu,\theta}(\eta)
\Big)^{\prime}\cdot J_i(\eta)\right\}Y_i^{\prime}(0)\cr
&+\sum_i\left\{\nabla_{(\rho_{\eta},U_{\eta},\mathcal{T}_{\nu,\theta\eta},T_{\theta\eta})}\mathcal{M}_{\nu,\theta}(\eta)
\cdot \big(J_i(\eta)\big)^{\prime}\right\}Y^{\prime}_i(0)\cr
&=I+II.
\end{align*}
Now, since we have
\begin{align}\label{now}
\begin{split}
\Big(\nabla_{(\rho_{\eta},U_{\eta},\mathcal{T}_{\nu,\theta\eta},T_{\theta\eta})}\mathcal{M}_{\nu,\theta}(\eta)
\Big)^{\prime}
&=
\nabla_{X_{\eta}} \left\{\nabla_{X_{\eta}} \mathcal{M}_{\nu,\theta}(\eta)\right\}
J(\eta)\big\{Y^{\prime}(\eta)\big\}^{\top}\cr
&=
\nabla_{X_{\eta}} \left\{\nabla_{X_{\eta}} \mathcal{M}_{\nu,\theta}(\eta)\right\}
J(\eta)\big\{Y^{\prime}(0)\big\}^{\top}\cr
&=(a_1,\cdots,a_{11}),
\end{split}
\end{align}
where
\begin{align*}
a_{k}&=\Big\{\nabla_{X_{\eta}} \left\{\nabla_{X_{\eta}} \mathcal{M}_{\nu,\theta}(\eta)\right\}
J(\eta)\Big\}_k \cdot Y^{\prime}(0)\cr
&=\sum_{\ell}\Big\{\nabla_{X_{\eta}} \left\{\nabla_{X_{\eta}} \mathcal{M}_{\nu,\theta}(\eta)\right\}
J(\eta)\Big\}_{k\ell}  Y^{\prime}_{\ell}(0),
\end{align*}
and
\begin{align}\label{where}
\begin{split}
J^{\prime}_i(\eta)&=\nabla_{X_{\eta}} J_i(\eta)J(\eta)\left\{Y^{\prime}(\eta)\right\}^{\top}\cr
&=\nabla_{X_{\eta}} J_i(\eta)J(\eta)Y^{\prime}(0)^{\top}\cr
&=\sum_j\nabla_{X_{\eta}} J_i(\eta)\cdot J_j(\eta) Y_j(0),
\end{split}
\end{align}
we can derive the following expression for $I$:
\begin{align*}
I&=\sum_i\left\{\Big(\sum_k\sum_{\ell}\Big\{\nabla_{X_{\eta}} \left\{\nabla_{X_{\eta}} \mathcal{M}_{\nu,\theta}(\eta)\right\}
J(\eta)\Big\}_{k\ell}  Y^{\prime}_{\ell}(0)
\Big) J_{ki}(\eta)\right\}Y_i^{\prime}(0)\cr
&=\sum_{i,k,\ell}A_{k\ell}J_{ki}(\eta)Y^{\prime}_{\ell}(0)Y^{\prime}_{i}(0)\cr
&=\sum_{i,k,\ell}A_{k\ell}J_{ki}(\eta)\langle f,e_{\ell}\rangle_{L^2_{v,I}}\langle f,e_i\rangle_{L^2_{v,I}}
\end{align*}
with
\[
A_{k\ell}=\Big\{\nabla_{X_{\eta}} \left\{\nabla_{X_{\eta}} \mathcal{M}_{\nu,\theta}(\eta)\right\}
J(\eta)\Big\}_{k\ell}.
\]
In view of Lemma \ref{matrix}  and Lemma \ref{22}, it can be easily verified that $A_{k\ell}J_{ki}(\eta)$ takes the following form:
\begin{align*}
\frac{P^{\mathcal{M}}_{i,j}(\rho_{\eta},U_{\eta}, v-U_{\eta}, \mathcal{T}^{-1}_{\nu,\theta \eta},I^{2/\delta}, T_{\theta\eta})}{R^{\mathcal{M}}_{i,j}(\rho_{\eta},\det\mathcal{T}_{\nu,\theta\eta},T_{\theta\eta})}
\mathcal{M}_{\nu,\theta}(\eta)
\end{align*}
for some polynomials $P^{\mathcal{M}}_{i,j}$, $R^{\mathcal{M}}_{i,j}$ satisfying the structural assumptions.
$II$ can be treated in a similar manner.
\end{proof}
%
%
%
%
%
%
%
Finally, we consider the linearization of the collision frequency.
%
%
%
%
\begin{lemma}\label{24}The collision frequency $A_{\nu,\theta}$ can be linearized around the normalized global Maxwellian as follows:
\begin{align*}
A_{\nu,\theta}&=\frac{1}{1-\nu+\theta\nu}\bigg\{1+\sum_{2\leq i\leq 7, 11}a_i\langle f,e_i\rangle_{L^2_{v,I}}\bigg\},
\end{align*}
where
\begin{align*}
a_1&=-\frac{2}{3+\delta}\int^1_0|U_{\eta}|^2d\eta,\quad
a_{i}=\frac{4}{3+\delta}\int^1_0U_{\eta,i}d\eta. \quad(i=2,3,4)\cr
a_{i}&=\frac{1}{3+\delta} \quad(i=5,6,7),\quad
a_{11}=\frac{\delta}{3+\delta}.
\end{align*}
\end{lemma}
\begin{proof}
We compute
\begin{align}\label{second}
\begin{split}
\rho T_{\delta}&=\int_{\mathbb{R}^3\times \mathbb{R}_+}F
\left(\frac{1}{3+\delta}|v-U|^2+\frac{2}{3+\delta}I^{2/\delta}\right)dvdI\cr
&=\int_{\mathbb{R}^3\times\mathbb{R}_+}(m+\sqrt{m}f)\left(\frac{1}{3+\delta}|v|^2+\frac{2}{3+\delta}I^{2/\delta}\right)dvdI+\frac{2}{3+\delta}\rho |U|^2\cr
&=1+\int_{\mathbb{R}^3\times\mathbb{R}_+}\sqrt{m}f\left(\frac{1}{3+\delta}|v|^2+\frac{2}{3+\delta}I^{2/\delta}\right)dvdI+\frac{2}{3+\delta}\rho |U|^2.
\end{split}
\end{align}
Then, observe
\[
|v|^2+2I^{2/\delta}=e_5+e_6+e_7+\delta e_{11}
\]
to write the second term in the last line of (\ref{second}) as
\[
\int_{\mathbb{R}^3\times\mathbb{R}_+}\sqrt{m}f\left(\frac{1}{3+\delta}|v|^2+\frac{2}{3+\delta}I^{2/\delta}\right)dvdI
=\frac{1}{3+\delta}\sum_{i=5,6,7}\langle f,e_i\rangle_{L^2_{v,I}}+\frac{\delta}{3+\delta}\langle f,e_{11}\rangle_{L^2_{v,I}}
\]
For the third term, we define
\begin{align*}
R(\eta)=\rho_{\eta}|U_{\eta}|^2 =\frac{|B(\eta)|^2}{A(\eta)}.
\end{align*}
Note that $R(1)=\rho |U|^2$ and $R(0)=0$.
Therefore, applying Taylor's theorem and the chain rule with
\[
A^{\prime}(\eta)=A^{\prime}(0),\quad B^{\prime}(\eta)=B^{\prime}(0),
\]
yields
\begin{align*}
R(1)&=R(0)+\int^1_0R^{\prime}(\eta)d\eta\cr
&=\int^1_0\left\{\frac{\partial R}{\partial A}A^{\prime}(0)+\frac{\partial R}{\partial B}B^{\prime}(0)\right\}d\eta\cr
&=-\left(\int^1_0\frac{|B(\eta)|^2}{|A(\eta)|^2}d\eta\right)A^{\prime}(0)
+\left(\int^1_0\frac{2B(\eta)}{A(\eta)}d\eta\right)\cdot B^{\prime}(0)\cr
&=-\left(\int^1_0|U_{\eta}|^2d\eta\right)\langle f,e_1\rangle_{L^2_{v,I}}
+\sum_{i=2,3,4}\left(\int^1_02U_{\eta,i}d\eta\right) \langle f,e_i\rangle_{L^2_{v,I}}
\end{align*}
This completes the proof.
\end{proof}
%
%
%
\subsection{Linearized polyatomic BGK model:}
Now we finish our linearization process.
To further simplify the presentation of the linearized relaxation operator, we denote
\begin{align*}
\mathcal{Q}^{\mathcal{M}}_{ij}(f)&=\frac{1}{\sqrt{m}}\int^1_0\frac{P^{\mathcal{M}}_{i,j}(\rho_{\eta}, v-U_{\eta}, \mathcal{T}^{-1}_{\nu,\theta \eta}
, I^{2/\delta},T_{\theta\eta})}{R^{\mathcal{M}}_{ij}(\rho_{\eta},\det\mathcal{T}_{\nu,\theta \eta},T_{\theta\eta})}
\mathcal{M}_{\eta}(\eta)
(1-\eta)d\eta,\cr
\mathcal{Q}^{A}_{i}(f)&=\left\{
\begin{array}{cl}
-\frac{2}{3+\delta}\int^1_0|U_{\eta}|^2d\eta  &  (i=1)\\
\frac{4}{3+\delta}\int^1_0U_{\eta,i}d\eta
& (i=2,3,4)\\
\frac{1}{3+\delta}, & (i=5,6,7) \\
\frac{\delta}{3+\delta}, & (i=11) \\
                         0 & \hbox{\mbox{otherwise}}
\end{array}
\right.\\
\end{align*}
so that we can represent the linearized operators more succinctly as
\begin{equation*}
\mathcal{M}_{\nu,\theta}(F)-F=\big(P_{\nu,\theta}f-f\big)\sqrt{m}+\sum_{ij} \mathcal{Q}^{\mathcal{M}}_{ij}\langle f,e_i\rangle_{L^2_{v,I}}\langle f,e_j\rangle_{L^2_{v,I}}\sqrt{m},
\end{equation*}
and
\begin{equation*}
A_{\nu,\theta}=\frac{1}{1-\nu+\theta\nu}\left\{1+\sum_{i}\mathcal{Q}^{A}_i\langle f,e_{i}\rangle_{L^2_{v,I}}\right\}.
\end{equation*}
We summarize all the computations of this section in the following proposition.
\begin{proposition}\label{Mf-f} The polyatomic relaxation operator can be linearized around the global polyatomic Maxwellian $m$ as follows
\begin{align*}
A_{\nu,\theta}\big(\mathcal{M}_{\nu,\theta}(F)-F\big)&=\frac{1}{1-\nu+\theta\nu}\Big(1+\sum_{i}\mathcal{Q}^{A}_{i}
\langle f,e_i\rangle_{L^2_{v,I}}\Big)\cr
&\times\Big\{\big(P_{\nu,\theta}f-f\big)+\sum_{i,j}\mathcal{Q}^{\mathcal{M}}_{ij}\langle f,e_i\rangle_{L^2_{v,I}}\langle f,e_j\rangle_{L^2_{v,I}}\Big\}\sqrt{m},
\end{align*}
where
\[
P_{\nu,\theta}f=\theta P_{p}f+(1-\theta)\{P_mf+\nu(P_1f+P_2f)\}.
\]
\end{proposition}
We now substitute $F=m+\sqrt{m}f$ into (1.1) and apply Proposition \ref{Mf-f} to obtain the perturbed polyatomic BGK model:
\begin{eqnarray}\label{LBGK}
\begin{split}
\partial_t f+v\cdot\nabla_x f&=L_{\nu,\theta}f+\Gamma_{\nu,\theta}(f),\cr
f(0,x,v,I)&=f_0(x,v,I),
\end{split}
\end{eqnarray}
where
\[
f_0(x,v,I)=\frac{F_0(x,v,I)-m(v,I)}{\sqrt{m(v,I)}}.
\]
The linearized relaxation operator $L_{\nu,\theta}$ is defined
as follows:
\begin{align*}
L_{\nu,\theta}f&=\frac{1}{1-\nu+\theta\nu}\left\{P_{\nu,\theta}f-f\right\},
\end{align*}
where the precise form of the polyatomic projection $P_{\nu,\theta}$ is stated in Theorem \ref{expansion},
 and the nonlinear perturbation $\Gamma_{\nu,\theta}(f)$ is given by
\begin{align*}
\Gamma_{\nu,\theta}(f)&=\sum_{i}\mathcal{Q}^{A}_{i}\langle f,e_i\rangle_{L^2_{v,I}}\left\{P_{\nu,\theta}f-f\right\}\cr
&+\sum_{i,j }\mathcal{Q}^{\mathcal{M}}_{i,j}\langle f,e_i\rangle_{L^2_{v,I}}\langle f,e_j\rangle_{L^2_{v,I}}\cr
&+\sum_{i,j,k}\mathcal{Q}^{A}_i\mathcal{Q}^{\mathcal{M}}_{j,k}\langle f,e_i\rangle_{L^2_{v,I}}
\langle f,e_j\rangle_{L^2_{v,I}}
\langle f,e_k\rangle_{L^2_{v,I}}\cr
&\equiv\Gamma_1(f,f)+\Gamma_{2}(f,f)+\Gamma_3(f,f,f).
\end{align*}

The conservation laws in (\ref{ConservationLawsF}) now take the following form:
\begin{eqnarray}\label{ConservationLawsf0}
\begin{split}
\int f(t)\sqrt{m} ~dxdvdI&=\int f_0\sqrt{m} ~dxdvdI,\cr
\int f(t)v\sqrt{m} ~dxdvdI&=\int f_0v\sqrt{m} ~dxdvdI,\cr
\int f(t)\left\{\frac{1}{2}|v|^2+I^{2/\delta}\right\}\sqrt{m} ~dxdvdI&
=\int f_0\left\{\frac{1}{2}|v|^2+I^{2/\delta}\right\}\sqrt{m} ~dxdvdI.
\end{split}
\end{eqnarray}
Therefore, if initial data shares the same mass, momentum and energy with $m$, the conservation laws read
\begin{eqnarray}\label{ConservationLawsf}
\begin{split}
\int_{\mathbb{T}^3_x\times\mathbb{R}^3_v\times\mathbb{R}^+_I}f(x,v,t)\sqrt{m} ~dxdvdI&=0,\cr
\int_{\mathbb{T}^3_x\times\mathbb{R}^3_v\times\mathbb{R}^+_I}f(x,v,t)v\sqrt{m} ~dxdvdI&=0,\cr
\int_{\mathbb{T}^3_x\times\mathbb{R}^3_v\times\mathbb{R}^+_I}f(x,v,t)\left\{\frac{1}{2}|v|^2+I^{2/\delta}\right\}\sqrt{m} ~dxdvdI&=0.
\end{split}
\end{eqnarray}
%
%

\section{Coercivity of the linearized relaxation operator}
The main goal of this section to establish the following  dissipative property of the linearized polyatomic relaxation operator. Note that the coefficient and the degeneracy in the right hand side see an abrupt jump as $\theta$ reaches 0.
\begin{theorem}\label{degenerate coercivity} Let  $-1/2<\nu<1$ and $0\leq \theta\leq 1$. Then we have the following dichotomy.
\begin{enumerate}
\item For $0< \theta\leq1$, $L_{\nu,\theta}$ satisfies
\begin{align*}
(1-\nu+\theta\nu)\langle L_{\nu,\theta}f,f\rangle_{L^2_{x,v,I}}
&\leq -\theta\|(I-P_p)f\|^2_{L^2_{x,v,I}}.
\end{align*}
\item If $\theta=0$, $L_{\nu,0}$ satisfies
\begin{align*}
(1-\nu)\langle L_{\nu,0}f,f\rangle_{L^2_{x,v,I}}
&\leq -\left(1-|\nu|\,\right)\|(I-P_m)f\|^2_{L^2_{x,v,I}}.
\end{align*}
\end{enumerate}
\end{theorem}
Before proving this theorem, we first need to establish several technical lemmas.

\begin{lemma}\label{Nilpotent}
The projection operators $P_p$, $P_m$, $P_1$, $P_2$ satisfy
\begin{enumerate}
\item $P_p$, $P_m$, $P_1$ and $P_2$ are orthogonal projections:
\begin{eqnarray*}
~P^2_p=P_p,~ P^2_m=P_m,~P^2_1=P_1,~P^2_2=P_2.
\end{eqnarray*}
\item $P_m$, $P_1$ and $P_2$ are mutually orthogonal in the following sense:
\begin{eqnarray*}
P_mP_1=P_1P_m=P_mP_2=P_2P_m=P_1P_2=P_2P_1=0.
\end{eqnarray*}
\end{enumerate}
\end{lemma}
\begin{proof}
(1) The first, second and the last identities follow from the fact that the elements of each of the following sets
\[
\bigg\{\sqrt{m },\,v\sqrt{m },\,\,\frac{(|v|^2-3)+(2I^{2\delta}-\delta)}{\sqrt{2(3+\delta)}}\sqrt{m }\bigg\},
\]
\[
\bigg\{\sqrt{m },\,v\sqrt{m },\,\,\frac{|v|^2-3}{\sqrt{6}}\sqrt{m },\,
\, \frac{2I^{2/\delta}-\delta}{\sqrt{2\delta}}\sqrt{m}\bigg\},
\]
\[
\left\{v_1v_2\sqrt{m },\,\, v_2v_3\sqrt{m },\,\, v_3v_1\sqrt{m }\right\}
\]
are orthonormal, which can be checked by a direct calculation.
The identity for $P_1$ needs more consideration. We first compute
\begin{eqnarray*}
&&\left\langle (3v^2_i-|v|^2)\sqrt{m }, (3v^2_i-|v|^2)\sqrt{m }\right\rangle_{L^2_{v,I}}=12,\hspace{0.6cm}(1=1,2,3)\cr
&&\left\langle (3v^2_i-|v|^2)\sqrt{m }, (3v^2_j-|v|^2)\sqrt{m }\right\rangle_{L^2_{v,I}}=-6 \quad\quad(i\neq j).
\end{eqnarray*}
Let us denote $c_i(v)=(3v^2_i-|v|^2)/3\sqrt{2}$, and use the above computations to see that
\begin{align*}
P_1^2f&=
P_1\big\{\langle f,c_1\rangle_{L^2_{v,I}} c_1+\langle f,c_2\rangle_{L^2_{v,I}} c_2+\langle f,c_3\rangle_{L^2_{v,I}} c_3\big\}\cr
&=
\langle f,c_1\rangle_{L^2_{v,I}} \big\{P_1c_1\big\}+\langle f,c_2\rangle_{L^2_{v,I}} \big\{P_2c_2\big\}+\langle f,c_3\rangle_{L^2_{v,I}} \big\{P_3c_3\big\}\cr
&=\frac{1}{3}\langle f,c_1\rangle_{L^2_{v,I}}\big\{2c_1-c_2-c_3\big\}\cr
&+\frac{1}{3}\langle f,c_2\rangle_{L^2_{v,I}}\big\{-c_1+2c_2-c_3\big\}\cr
&+\frac{1}{3}\langle f,c_3\rangle_{L^2_{v,I}}\big\{-c_1-c_2+2c_3\big\}.
\end{align*}
The last term can be rewritten as
\begin{eqnarray*}
\Big\langle f,\frac{2c_1-c_2-c_3}{3}\Big\rangle_{L^2_{v,I}} c_1+\Big\langle f,\frac{-c_1+2c_2-c_3}{3}\Big\rangle_{L^2_{v,I}} c_2+\Big\langle f,\frac{-c_1-c_2+2c_3}{3}
\Big\rangle_{L^2_{v,I}} c_3,
\end{eqnarray*}
which, in view of $c_1+c_2+c_3=0$, is
\begin{eqnarray*}
\langle f,c_1\rangle_{L^2_{v,I}} c_1+\langle f,c_2\rangle_{L^2_{v,I}} c_2+\langle f,c_3\rangle_{L^2_{v,I}} c_3.
\end{eqnarray*}
Therefore, we have $P_1^2f=P_1f$.\newline
(2) We observe from direct computation that the following quantities all vanish:
\begin{eqnarray*}
&&\langle \big( 2I^{2/\delta}-\delta)\big)\sqrt{m},\sqrt{m}\,\rangle_{L^2_{v,I}}, ~
\langle \big( 2I^{2/\delta}-\delta)\big)\sqrt{m},v_{\ell}\sqrt{m }\,\rangle_{L^2_{v,I}},\cr
&&\langle \big( 2I^{2/\delta}-\delta)\big)\sqrt{m},\big(|v|^2-3\big)\sqrt{m }\,\rangle_{L^2_{v,I}},~
\langle \sqrt{m },(3v^2_i-|v|^2)\sqrt{m }\,\rangle_{L^2_{v,I}},\cr
&&\langle v_{\ell}\sqrt{m },(3v^2_i-|v|^2)\sqrt{m }\,\rangle_{L^2_{v,I}},~
\langle (|v|^2-3)\sqrt{m }, (3v^2_i-|v|^2)\sqrt{m }\,\rangle_{L^2_{v,I}},\cr
&&\langle v_iv_j\sqrt{m }, (3v^2_k-|v|^2)\sqrt{m }\rangle_{L^2_{v,I}}=0,~
\langle v_iv_j\sqrt{m }, \big( 2I^{2/\delta}-\delta)\big)\sqrt{m}\rangle_{L^2_{v,I}}=0,
\end{eqnarray*}
which implies (2).
\end{proof}

\begin{lemma}\label{degenerate coercivity0} For $0\leq \theta \leq 1$, we have
\begin{align*}
&-(1-\nu+\theta\nu)\langle L_{\nu,\theta}f,f\rangle_{L^2_{x,v,I}}\cr
&\qquad= \theta\|(I-P_p)f\|^2_{L^2_{x,v,I}}
+(1-\theta)\left\{\|(I-P_m)f\|^2_{L^2_{x,v,I}}-\nu\|(P_1+P_2)f\|^2_{L^2_{x,v,I}}\right\}.
\end{align*}
\end{lemma}
\begin{proof}
From the definition of $L_{\nu,\theta}$, we have
\begin{eqnarray}\label{setting}
\begin{split}
&-(1-\nu+\theta\nu)\langle L_{\nu,\theta}f,f\rangle_{L^2_{v}}\cr
&\qquad=-\theta\langle P_{p}f-f,f\rangle_{L^2_{v,I}}
-(1-\theta)\langle P_mf-f+\nu(P_1+P_2)f,f\rangle_{L^2_{v,I}}\cr
&\qquad\equiv \theta I+(1-\theta)II.
\end{split}
\end{eqnarray}
Then, the desired result follows from  (1) and (2) below.\newline
\noindent(1) {\bf The estimate of I:} Lemma \ref{Nilpotent} (1) immediately gives
\[
\langle P_{p}f-f,f\rangle_{L^2_{v,I}}=-\|(I-P_{p})f\|^2_{L^2_{v,I}}.
\]
(2) {\bf The estimate of II:}
As in the previous case, we have from Lemma \ref{Nilpotent} (1)
\[
\langle P_{m}f-f,f\rangle_{L^2_{v,I}}=-\|(I-P_{m}f)\|^2_{L^2_{v,I}}.
\]
On the other hand, we observe from Lemma \ref{Nilpotent} that
\begin{align*}
(P_1+P_2)^2&=P^2_1+P_1P_2+P_2P_1+P_2^2\cr
&=P_1+P_2,
\end{align*}
to derive
\begin{align*}
\langle(P_1+P_2)f,f\rangle_{L^2_{v,I}}&=\langle(P_1+P_2)^2f,f\rangle_{L^2_{v,I}}\cr
&=\langle(P_1+P_2)f,(P_1+P_2)f\rangle_{L^2_{v,I}}\cr
&=\|(P_1+P_2)f\|^2_{L^2_{v,I}},
\end{align*}
where we used the symmetry of $P_1+P_2$.
We combine these estimates to derive
\begin{align*}
II&=-\langle P_mf-f+\nu(P_1+P_2)f,f\rangle_{L^2_{v,I}}\cr
&=-\langle P_mf-f,f\rangle_{L^2_{v,I}}-\nu\langle(P_1+P_2)f,f\rangle_{L^2_{v,I}}\cr
&= \|(I-P_{m}f)\|^2_{L^2_{v,I}}-\nu\|(P_1+P_2)f\|^2_{L^2_{v,I}}.
\end{align*}
\end{proof}
%
%
%
%
\noindent An immediate but important ramification of the above dissipation estimate is that the null space of the linearized relaxation operator
has the following dichotomy:
\begin{proposition}\label{Kernel L} For $0\leq \theta\leq 1$ and $-1/2<\nu<1$, the kernel of the linearized relaxation operator is given by
\begin{eqnarray*}
Ker\{L_{\nu,\theta}\}=span\bigg\{\sqrt{m },~v\sqrt{m },~\frac{(|v|^2-3)+(2I^{2/\delta}-\delta)}{\sqrt{2(3+\delta)}}\sqrt{m}\bigg\}, \quad (\theta\neq0)
\end{eqnarray*}
and
\begin{eqnarray*}
Ker\{L_{\nu,0}\}=span\bigg\{\sqrt{m },~v\sqrt{m },~\frac{|v|^2-3}{\sqrt{6}}\sqrt{m},~ \frac{I^{2/\delta}-\delta}{\sqrt{2\delta}}\sqrt{m }\bigg\}.
\qquad (\theta=0)
\end{eqnarray*}
\end{proposition}
\begin{proof}
For simplicity, set
\begin{align*}
A(f)&=\|(I-P_p)f\|^2_{L^2_{x,v,I}}\cr
B(f)&=\|(I-P_m)f\|^2_{L^2_{x,v,I}}-\nu\|(P_1+P_2)f\|^2_{L^2_{x,v,I}},
\end{align*}
so that, in view of Lemma \ref{degenerate coercivity0}, we write
\begin{align}\label{remove}
-(1-\nu+\theta\nu)\langle L_{\nu,\theta}f,f\rangle_{L^2_{x,v,I}}= \theta A(f)+(1-\theta)B(f).
\end{align}
The non-negativity of $A(f)$ is clear. We claim that it's the case for $B(f)$ too:\newline

\noindent$\bullet$ {\bf Clam}: $B(f)\geq 0$ for $-1/2<\nu<1$.\newline\newline
{\bf Proof of the claim:}
Lemma \ref{Nilpotent} says  $(P_1+ P_2)\perp P_m$, so that
\begin{align}\label{setting2}
\|(P_1+P_2)f\|_{L^2_{v,I}}
=\|(P_1+P_2)(I-P_m)f\|^2_{L^2_{v,I}}.
\end{align}
Then, since $(P_1+P_2)^2=P_1+P_2$, we see that
\begin{align}
\begin{split}\label{sincee}
\|(P_1+P_2)(I-P_m)f\|^2_{L^2_{v,I}}&=\langle(P_1+P_2)(I-P_m)f, (P_1+P_2)(I-P_m)f \rangle_{L^2_{v,I}}\cr
&=\langle(P_1+P_2)^2(I-P_m)f, (I-P_m)f \rangle_{L^2_{v,I}}\cr
&=\langle(P_1+P_2)(I-P_m)f,(I-P_m)f \rangle_{L^2_{v,I}}\cr
&\leq\|(P_1+P_2)(I-P_m)f\|_{L^2_{v,I}} \|(I-P_m)f \|_{L^2_{v,I}}.
\end{split}
\end{align}
Therefore, (\ref{setting2}) and (\ref{sincee}) gives
\begin{align}\label{setting3}
\|(P_1+P_2)f\|_{L^2_{v,I}}\leq \|(I-P_m)f \|_{L^2_{v,I}}.
\end{align}
Hence, we have
\begin{align}\label{setttt}
\begin{split}
B(f)&\geq\|(I-P_{m})f\|^2_{L^2_{v,I}}-|\nu|\|(I-P_m)f\|^2_{L^2_{v,I}}\cr
&= (1-|\nu|)\|(I-P_{m}f)\|^2_{L^2_{v,I}}\cr
&\geq 0.
\end{split}
\end{align}
This proves the claim.\newline

Now we return to the proof of the proposition. Consider
\begin{align}\label{fromfrom}
L_{\nu,\theta}f=\theta A(f)+(1-\theta)B(f)=0.
\end{align}
We divide it into the following two cases:\newline

\noindent(1) (The case $\theta=0$): In this case, (\ref{fromfrom}) reduces to
\begin{align*}
B(f)=0.
\end{align*}
That is,
\begin{align*}
\|(I-P_m)f\|^2_{L^2_{x,v,I}}+\nu\|(P_1+P_2)f\|^2_{L^2_{x,v,I}}=0,
\end{align*}
which, in view of (\ref{setting3}), implies
\begin{align*}
\|(I-P_m)f\|^2_{L^2_v}=-\nu\|(P_1+P_2)f\|_{L^2_v}
\leq|\nu|\|(I-P_m)f\|_{L^2_v},
\end{align*}
Therefore,
\begin{align*}
(1-|\nu|)\|(I-P_m)f\|^2_{L^2_{x,v,I}}\leq0,
\end{align*}
so that
\begin{eqnarray}\label{second term}
f=P_mf.
\end{eqnarray}
In conclusion, when $\theta=0$, we have
\[
Ker\{L_{\nu,0}\}=span\bigg\{\sqrt{m },~v\sqrt{m },~\frac{|v|^2-3}{\sqrt{6}}\sqrt{m},~ \frac{I^{2/\delta}-\delta}{\sqrt{2\delta}}\sqrt{m }\bigg\}.
\]

\noindent(2) (The case $\theta\neq0$): Since both $A$ and $B$ are non-negative, we have from (\ref{fromfrom})
\begin{align*}
A(f)=B(f)=0.
\end{align*}
First,
\begin{align*}
A(f)=\|(I-P_{p})f\|^2_{L^2_{x,v,I}}=0,
\end{align*}
clearly gives
\begin{eqnarray*}
f=P_pf.
\end{eqnarray*}
On the other hand, it was shown in the previous case that $B(f)=0$ implies
\begin{eqnarray*}
f=P_mf.
\end{eqnarray*}
Therefore,  when $0<\theta\leq1$, we have
\begin{align*}
f=P_{p}f=P_mf.
\end{align*}
Hence, the kernel is given by the intersection of
\begin{align*}
span\bigg\{\sqrt{m },~v\sqrt{m },~\frac{(|v|^2-3)+(2I^{2/\delta}-\delta)}{\sqrt{2(3+\delta)}}\sqrt{m}\bigg\}
\end{align*}
and
\begin{align*}
span\bigg\{\sqrt{m },~v\sqrt{m },~\frac{|v|^2-3}{\sqrt{6}}\sqrt{m},~ \frac{I^{2/\delta}-\delta}{\sqrt{2\delta}}\sqrt{m }\bigg\},
\end{align*}
This gives the desired result since the former is a subspace of the latter.
\end{proof}
%
%
%
%
\noindent We are now ready to prove the main theorem of this section.
\noindent\subsection{\bf Proof of Theorem \ref{degenerate coercivity}:}
\noindent(1) (The case of $0<\theta\leq 1$): In the proof of Proposition \ref{Kernel L}, we have shown that the degeneracy of
$B(f)$ is strictly bigger than that of $A(f)$. Therefore, we can ignore $B$ in (\ref{remove}) to obtain
\[
-(1-\nu+\theta\nu)\langle L_{\nu,\theta}f,f\rangle_{L^2_{v,I}}\geq \theta\|(I-P_p)f\|^2_{L^2_{v,I}}.
\]
\noindent(2) (The case of $\theta=0$): In this case, we are left with
\begin{align*}
-(1-\nu)\langle L_{\nu,0}f,f\rangle_{L^2_{v,I}}=B(f).
\end{align*}
Recall that we have shown in (\ref{setttt}) that
\[
B(f)\geq (1-|\nu|)\|(I-P_m)f\|^2_{L^2_{v,I}},
\]
to see
\begin{align*}
-(1-\nu)\langle L_{\nu,0}f,f\rangle_{L^2_{v,I}}\geq(1-|\nu|)\|(I-P_m)f\|^2_{L^2_{v,I}}.
\end{align*}
This completes the proof.\newline


\section{Estimates on the macroscopic fields}
\subsection{Estimates on the macroscopic fields}
In this section, we establish estimates on the macroscopic fields which will be crucially used to control the nonlinear term $\Gamma_{\nu,\theta}(f)$.
\begin{lemma}\label{ULestimatesofRhoUT}
Assume $\mathcal{E}(t)$ is sufficiently small, then there exists a positive constant $C>0$  such that
\begin{eqnarray*}
&&(1)~|\rho_{\eta}(x,t)-1|\leq C\sqrt{\mathcal{E}(t)},\cr
&&(2)~|U^i_{\eta }(x,t)|\leq C\sqrt{\mathcal{E}(t)}, \hspace{1.81cm}(1\leq i\leq 3)\cr
&&(3)~|\mathcal{T}_{\nu,\theta\eta}^{ii}(x,t)-1|\leq C\sqrt{\mathcal{E}(t)},\hspace{0.89cm}(1\leq i\leq 3)\cr
&&(4)~|\mathcal{T}^{ij}_{\nu,\theta \eta}(x,t)|\leq   C
\sqrt{\mathcal{E}(t)},\hspace{1.49cm}(1\leq i<j\leq 3)\cr
&&(5)~|T_{\theta \eta}(x,t)-1|\leq  C\sqrt{\mathcal{E}(t)}.
\end{eqnarray*}
\end{lemma}
\begin{proof}
(1) Since
\[
|F_{\eta}-m|=|\eta f\sqrt{m}|\leq |f|\sqrt{m},
\]
H\"{o}lder inequality and Sobolev embedding yield
\begin{eqnarray*}
|\rho_{\eta}(x,t)-1|=\int |f|\sqrt{m }dvdI\leq C\|f\|_{L^2_{v,I}}\leq C\sqrt{\mathcal{E}(t)}.
\end{eqnarray*}
(2) Note that
\[
\int F_{\eta}v dvdI=\int (m+\eta f\sqrt{m})v dvdI=\eta\int fv\sqrt{m}dvdI.
\]
Therefore, recalling the lower bound estimate of $\rho_{\eta}$ in (1), and employing H\"{o}lder inequality and Sobolev embedding, we have
\begin{align*}
|U_{\eta}|\leq \frac{1}{\rho_{\eta}}\left|\int_{\mathbb{R}^3} fv\sqrt{m }dvdI\right|\leq\frac{C\|f\|_{L^2_{v,I}}}{1-\sqrt{\mathcal{E}(t)}}
\leq C\sqrt{\mathcal{E}(t)}
\end{align*}
for sufficiently small $\mathcal{E}(t)$.\newline
(3) We recall  (\ref{TMF}) to write
the diagonal elements of $\rho\mathcal{T}_{\nu,\theta,\eta}$ as
 \begin{align*}
&\hspace{-0.3cm}\rho_{\eta}\mathcal{T}^{ii}_{\nu,\theta\eta}\cr
&=\theta\left[\left\{\int_{\mathbb{R}^3\times \mathbb{R}_+}F_{\eta}\left(\frac{1}{3+\delta}|v|^2+\frac{2}{3+\delta}I^{2/\delta}\right)
dvdI\right\}-\frac{1}{3+\delta}\rho_{\eta}|U_{\eta}|^2\right]\cr
&+(1-\theta)\left[\left\{\int_{\mathbb{R}^3\times \mathbb{R}_+}F_{\eta}\left(\frac{1-\nu}{3}|v|^2+\nu v^2_i\right)dvdI\right\}-\left\{\frac{1-\nu}{3}\rho_{\eta}|U_{\eta}|^2+\nu\rho_{\eta} U^2_{\eta,i}\right\}\right].
\end{align*}
Therefore, using
\[
\int m\left(\frac{1}{3+\delta}|v|^2+\frac{2}{3+\delta}I^{2/\delta}\right)dvdI=1,
\]
we have
 \begin{align*}
&\hspace{-0.3cm}\rho_{\eta}\mathcal{T}^{ii}_{\nu,\theta\eta}\cr
&=1+\theta\left[\left\{\int_{\mathbb{R}^3\times \mathbb{R}_+}\eta f\sqrt{m}\left(\frac{1}{3+\delta}|v|^2+\frac{2}{3+\delta}I^{2/\delta}\right)
dvdI\right\}-\frac{\theta}{3+\delta}\rho_{\eta}|U_{\eta}|^2\right]\cr
&+(1-\theta)\left[\left\{\int_{\mathbb{R}^3\times \mathbb{R}_+}\eta f\sqrt{m}\left(\frac{1-\nu}{3}|v|^2+\nu v^2_i\right)dvdI\right\}-\left\{\frac{1-\nu}{3}\rho_{\eta}|U_{\eta}|^2+\nu\rho_{\eta} U^2_{\eta,i}\right\}\right]\cr
&\leq1+\theta\left\{\int_{\mathbb{R}^3\times \mathbb{R}_+}\eta f\sqrt{m}\left(\frac{1}{3+\delta}|v|^2+\frac{2}{3+\delta}I^{2/\delta}\right)
dvdI\right\}\cr
&+(1-\theta)\left\{\int_{\mathbb{R}^3\times \mathbb{R}_+}\eta f\sqrt{m}\left(\frac{1-\nu}{3}|v|^2+\nu v^2_i\right)dvdI\right\}\cr
&\leq 1+C_{\theta,\delta}\|f\|_{L^2_{v,I}}.
\end{align*}
In the last line, we used H\"{o}lder inequality. Now, the estimate (1) above on $\rho_{\eta}$ and Sobolev embedding gives
\begin{align}\label{ThetaUpper}
\mathcal{T}^{ii}_{\nu,\theta\eta}-1&\leq \frac{1-\rho_{\theta}+C_{\theta,\delta}\|f\|_{L^2_{v,I}}}{\rho_{\theta}}\cr
&\leq \frac{C\sqrt{\mathcal{E}(t)}+C_{\theta,\delta}\|f\|_{L^2_{v,I}}}{1-\sqrt{\mathcal{E}(t)}}\cr
&\leq C_{\theta,\delta}\sqrt{\mathcal{E}(t)}.
\end{align}
Similarly, we compute
\begin{align*}
&\hspace{-0.3cm}\rho_{\eta}\mathcal{T}^{ii}_{\nu,\theta\eta}\cr
&=1+\theta\left[\left\{\int_{\mathbb{R}^3\times \mathbb{R}_+}\eta f\sqrt{m}\left(\frac{1}{3+\delta}|v|^2+\frac{2}{3+\delta}I^{2/\delta}\right)
dvdI\right\}-\frac{1}{3+\delta}\rho_{\eta}|U_{\eta}|^2\right]\cr
&+(1-\theta)\left[\left\{\int_{\mathbb{R}^3\times \mathbb{R}_+}\eta f\sqrt{m}\left(\frac{1-\nu}{3}|v|^2+\nu v^2_i\right)dvdI\right\}-\left\{\frac{1-\nu}{3}\rho_{\eta}|U_{\eta}|^2+\nu\rho_{\eta} U^2_{\eta,i}\right\}\right]\cr
&\geq1-\theta \left\{C_{\delta}\|f\|_{L^2_{v,I}}
+\frac{\theta}{3+\delta}\frac{\|f\|^2_{L^2_{v,I}}}{1-\sqrt{\mathcal{E}(t)}}\right\}
-(1-\theta)\left\{\eta C_{\nu}\|f\|_{L^2_{v,I}}
+C_{\nu}\frac{\|f\|^2_{L^2_{v,I}}}{1-\sqrt{\mathcal{E}(t)}}\right\}\cr
&\geq1-C_{\theta,\delta,\nu}\|f\|^2_{L^2_{v,I}}\cr
&\geq1-C_{\theta,\delta,\nu}\sqrt{\mathcal{E}(t)},
\end{align*}
yielding
\begin{align}\label{ThetaLower}
\begin{split}
\mathcal{T}^{ii}_{\nu,\theta\eta}-1&\geq\frac{1-\rho_{\theta}-C_{\theta,\delta,\nu}\sqrt{\mathcal{E}(t)}}{\rho_{\theta}}\cr
&\geq\frac{-C_{\theta,\delta,\nu}\sqrt{\mathcal{E}(t)}}{1+\sqrt{\mathcal{E}(t)}}\cr
&\geq -C_{\theta,\delta,\nu}\sqrt{\mathcal{E}(t)}.
\end{split}
\end{align}
(\ref{ThetaUpper}) and (\ref{ThetaLower}) give the desired result for $\mathcal{T}_{\nu,\theta\eta}^{ii}$ $(i=1,2,3)$.\newline

\noindent(4) Non-diagonal entries of $\mathcal{T}_{\nu,\theta\eta}$ are given by
\begin{align*}
\rho_{\eta}\mathcal{T}^{ij}_{\nu,\theta\eta}&=
(1-\theta)\left\{\int_{\mathbb{R}^3\times \mathbb{R}_+}F_{\eta}\left(\nu v_iv_j\right)dvdI-\nu\rho_{\eta} U^i_{\eta}U^j_{\eta}\right\}\cr
&=(1-\theta)\left\{\int_{\mathbb{R}^3\times \mathbb{R}_+}\big(m+\eta f\sqrt{m}\big)\left(\nu v_iv_j\right)dvdI-\nu\rho_{\eta} U^i_{\eta}U^j_{\eta}\right\}\cr
&=(1-\theta)\left\{ \eta\nu \int_{\mathbb{R}^3\times \mathbb{R}_+}f\sqrt{m}v_iv_jdvdI-\nu\rho_{\eta} U^i_{\eta}U^j_{\eta}\right\}.
\end{align*}
Hence
\begin{align*}
\big|\rho_{\eta}\mathcal{T}^{ij}_{\nu,\theta\eta}\big|
&\leq |\nu|\Big|\int_{\mathbb{R}^3\times\mathbb{R}^+}fv_iv_j\sqrt{m }dvdI\Big|+|\nu|\rho_{\eta}|U^i_{\eta}||U^j_{\eta}|\cr
&\leq |\nu| C\|f\|_{L^2_{v,I}}+|\nu| C\mathcal{E}(t)\cr
&\leq|\nu| C\sqrt{\mathcal{E}(t)}.
\end{align*}
(5) The estimate follows by similar argument using the following identity
\begin{align*}
&T_{\theta\eta}=\theta\left\{\frac{1}{\rho_{\eta}}\int_{\mathbb{R}^3\times \mathbb{R}_+}F_{\eta}\left(\frac{1}{3+\delta}|v|^2+\frac{2}{3+\delta}I^{2/\delta}\right)dvdI
-\frac{1}{3+\delta}|U_{\eta}|^2\right\}\cr
&\qquad+(1-\theta)\left\{\frac{2}{\delta}\frac{1}{\rho_{\eta}}\int_{\mathbb{R}^3\times \mathbb{R}_+}F_{\eta}I^{2/\delta}dvdI\right\}.
\end{align*}
We omit it.
\end{proof}
%
%
%
%
%
In the following, we estimate the derivatives of the macrosopic fields.
\begin{lemma}\label{ULestimatesofRhoUT1}
For sufficiently small $\mathcal{E}(t)$, we have
\begin{eqnarray*}
&&(1)~|\partial^{\alpha}\rho_{\eta}(x,t)|\leq C_{\alpha}\sqrt{\mathcal{E}(t)},\cr
&&(2)~|\partial^{\alpha}U_{\eta}(x,t)|\leq C_{\alpha}\sqrt{\mathcal{E}(t)},\cr
&&(3)~|\partial^{\alpha}\mathcal{T}_{\nu,\theta\eta}^{ij}(x,t)|\leq C_{\alpha}\sqrt{\mathcal{E}(t)},\cr
&&(4)~|\partial^{\alpha}T_{\theta\eta}(x,t)|\leq C_{\alpha}\sqrt{\mathcal{E}(t)}.
\end{eqnarray*}
Here $\partial$ denotes derivatives in $x, t$.
\end{lemma}
\begin{proof}
(1) Since $\partial^{\alpha}\!\int mdvdI =0$, we have
\begin{align*}
|\partial^{\alpha}\rho_{\eta}|&=\left|\partial^{\alpha}\left(\int_{\mathbb{R}^3\times\mathbb{R}_+} \big(m +\eta f\sqrt{m }\big)dvdI\right)\right|=\eta\int |\partial^{\alpha} f|\sqrt{m }dvdI
\leq\|\partial^{\alpha}f\|_{L^2_{v,I}}.
\end{align*}
(2) We apply $\partial^{\alpha}$ to $U=\frac{1}{\rho}\int f v\sqrt{m }dv$ and use Leibniz rule to derive
\begin{eqnarray*}
\displaystyle|\partial^{\alpha}U_{\eta}|\leq
\frac{C_{|\alpha|}}{\rho_{\eta}^{2|\alpha|}}
\left(\sum_{|\alpha_1|\leq N}\int_{\mathbb{R}^3\times\mathbb{R}_+}|\partial^{\alpha_1}f||v|\sqrt{m }dvdI\right)\left(1+\sum_{|\alpha_2|\leq N}|\partial^{\alpha_2}\rho_{\eta}|\right)^{|\alpha|}.
\end{eqnarray*}
Then, we have from H\"{o}lder inequality and Sobolev embedding that
\begin{align*}
\displaystyle|\partial^{\alpha}U_{\eta}|&\leq
\frac{C_{\alpha}}{(1-\mathcal{E}(t))^{2|\alpha|}}
\left(\sum_{|\alpha_1|\leq N}\|\partial^{\alpha_1}f\|_{L^2_{v,I}}\right)\left(1+\sum_{|\alpha_2|\leq N}
\|\partial^{\alpha_2}f\|_{L^2_{v,I}}\right)^{|\alpha|}\cr
&\leq C_{\alpha}\left\{\sum_{|\alpha_1|\leq N}\|\partial^{\alpha_1}f\|_{L^2_{v,I}}
+\bigg(\sum_{|\alpha_1|\leq N}\|\partial^{\alpha_1}f\|_{L^2_{v,I}}\bigg)^{|\alpha|}\right\}\cr
&\leq C_{\alpha}\sqrt{\mathcal{E}(t)}.
\end{align*}
for sufficiently small $\mathcal{E}(t)$.\newline

(3) Similar argument as in (2) above, applied to $\rho\mathcal{T}_{\nu,\theta\eta}$ gives
\begin{align*}
&|\partial^{\alpha}\mathcal{T}^{ij}_{\nu,\theta\eta}|\cr
&\quad\leq
\frac{C_{|\alpha|,\delta,\theta}}{\rho^{2|\alpha|}_{\eta}}
\left(\sum_{|\alpha_1|\leq N}\int_{\mathbb{R}^3\times\mathbb{R}_+}|\partial^{\alpha_1}f|\left\{|v|^2+2I^{\delta/2}\right\}\sqrt{m }dvdI\right)\left(1+\sum_{|\alpha_2|\leq N}|\partial^{\alpha_2}\rho_{\eta}|\right)^{|\alpha|}\cr
&\quad+
\frac{C_{|\alpha|}}{\rho_{\eta}^{2|\alpha|}}
\left(\sum_{|\alpha_1|\leq N}\int_{\mathbb{R}^3\times\mathbb{R}_+}|\partial^{\alpha_1}f||v|\sqrt{m }dvdI\right)\left(1+\sum_{|\alpha_2|\leq N}|\partial^{\alpha_2}\rho_{\eta}|\right)^{|\alpha|}
\cr
&\quad\leq
 C_{|\alpha|}\left\{\sum_{|\alpha_1|\leq N}\|\partial^{\alpha_1}f\|_{L^2_{v,I}}
+\bigg(\sum_{|\alpha_1|\leq N}\|\partial^{\alpha_1}f\|_{L^2_{v,I}}\bigg)^{|\alpha|}\right\}.
\end{align*}
(4) The  estimate for $\partial^{\alpha} T_{\theta\eta}$ is similar. We omit it.
\end{proof}

%
%
%
%
%
\begin{lemma}\label{Det_Estimate}Let $\mathcal{E}(t)$ be sufficiently small. Then the determinant of $\mathcal{T}_{\nu,\theta\eta}$ satisfies
\begin{eqnarray*}
&&(1)~  \big|\partial^{\alpha}\det\big(\mathcal{T}_{\nu,\theta\eta}\big)\big|\leq C_{\alpha}\sqrt{\mathcal{E}(t)},\cr
&&(2)~ \big|\det\big(\mathcal{T}_{\nu,\theta\eta}\big)\big|\geq 1-C_{\alpha}\sqrt{\mathcal{E}(t)},
\end{eqnarray*}
for some $C_{\alpha}>0$.
\end{lemma}
\begin{proof}
(1) A straightforward computation gives
\begin{align}\label{deter}
\begin{split}
\det(\mathcal{T}_{\nu,\theta\eta})&=\mathcal{T}^{11}_{\nu,\theta\eta}\mathcal{T}^{22}_{\nu,\theta\eta}\mathcal{T}^{33}_{\nu,\theta\eta}
+2\mathcal{T}^{12}_{\nu,\theta\eta}\mathcal{T}^{23}_{\nu,\theta\eta}\mathcal{T}^{31}_{\nu,\theta\eta}\cr
&-\big\{\mathcal{T}^{23}_{\nu,\theta\eta}\big\}^2\mathcal{T}^{11}_{\nu,\theta\eta}
-\big\{\mathcal{T}^{31}_{\nu,\theta\eta}\big\}^2\mathcal{T}^{22}_{\nu,\theta\eta}
-\big\{\mathcal{T}^{12}_{\nu,\theta\eta}\big\}\mathcal{T}^{33}_{\nu,\theta\eta}.
\end{split}
\end{align}
Therefore,  $\partial^{\alpha}\det\mathcal{T}_{\nu,\theta\eta}$ takes the following form:
\begin{eqnarray*}
\partial^{\alpha}\det(\mathcal{T}_{\nu,\theta\eta})=\sum_{\alpha=\alpha_1+\alpha_2+\alpha_3}
C_{ij\ell kmn}
\partial^{\alpha_1}\mathcal{T}_{\nu,\theta\eta}^{ij}\partial^{\alpha_2}\mathcal{T}^{\ell k}_{\nu,\theta \eta}\partial^{\alpha_3 }\mathcal{T}^{mn}_{\nu,\theta\eta}.
\end{eqnarray*}
Now, we recall from Lemma \ref{ULestimatesofRhoUT} and Lemma \ref{ULestimatesofRhoUT1} that
\begin{align}\label{44}
\mathcal{T}^{ii}_{\nu,\theta\eta }=1+o\big(\sqrt{\mathcal{E}(t)}\,\big)~(i=1,2,3),
\quad\mathcal{T}^{ij}_{\nu,\theta\eta }=o~\big(\sqrt{\mathcal{E}(t)}\,\big) ~(i\neq j),
\end{align}
and
\begin{align}\label{45}
\big|\partial^{\alpha}\mathcal{T}^{ij}_{\nu,\theta\eta }\big|\leq C_{\alpha}\sqrt{\mathcal{E}(t)},
\end{align}
to deduce
\begin{align*}
|\partial^{\alpha}\det\mathcal{T}_{\nu,\theta\eta}|
&\leq C_{\alpha}\sqrt{\mathcal{E}(t)}.
\end{align*}
(2) Inserting (\ref{44}) and (\ref{45}) into (\ref{deter}), we get
\begin{align*}
\det\mathcal{T}_{\nu,\theta\eta}&=\big\{1+o(\mathcal{E}(t))\big\}^3
-2\left\{o\big(\sqrt{\mathcal{E}(t)}\,\big)\right\}^3-3\left\{o\big(\sqrt{\mathcal{E}(t)}\,\big)\right\}^2
\left\{1+o\big(\sqrt{\mathcal{E}(t)}\,\big)\right\}\cr
&\geq 1-C_{\alpha}\sqrt{\mathcal{E}(t)},
\end{align*}
for sufficiently small $\mathcal{E}(t)$.
\end{proof}

%
%
%
%
\begin{lemma}\label{UniformEst}
Let $0\leq \theta\leq 1$ and $-1/2<\nu<1$. Suppose $\mathcal{E}(t)$ is sufficiently small. Then, there exist positive constants $C_1, C_2$ such that
\begin{align*}
&(1)~X^{\top}\{\mathcal{T}^{-1}_{\nu,\theta\eta}\}Y\leq \{1-C_1\mathcal{E}(t)\}^{-1}\|X\|\|Y\|,\cr
&(2)~X^{\top}\{\mathcal{T}^{-1}_{\nu,\theta\eta}\}X\geq \{1+C_2\mathcal{E}(t)\}^{-1}\|X\|^2,
\end{align*}
for any $X$, $Y$ in $\mathbb{R}^3$.
\end{lemma}
\begin{proof}
We start with proving the following claim:
\newline{\bf $\bullet$ Claim:} For sufficiently small $\mathcal{E}(t)$, we have
\begin{align}\label{following claim}
\left\{1-C_1\sqrt{\mathcal{E}(t)}\right\}Id\leq \mathcal{T}_{\nu,\theta\eta}\leq \left\{1+C_2\sqrt{\mathcal{E}(t)}\right\}Id.
\end{align}
{\bf Proof of the claim:} For $\kappa\in\mathbb{R}^3$, we have
\begin{align*}
\kappa^{\top} \mathcal{T}_{\nu,\theta\eta}\kappa=\sum_{i=1,2,3}\mathcal{T}_{\nu,\theta\eta}^{ii}\kappa^2_i+\sum_{1\leq i,j\leq3}
\mathcal{T}_{\nu,\theta\eta}^{ij}\kappa_i\kappa_j.
\end{align*}
In view of Lemma \ref{ULestimatesofRhoUT} (3), (4), this gives
\begin{align*}
\kappa^{\top} \mathcal{T}_{\nu,\theta\eta}\kappa&=\sum_{i=1,2,3}\big\{1+C\sqrt{\mathcal{E}(t)}\big\}\kappa^2_i+C\sum_{1\leq i,j\leq3}
\sqrt{\mathcal{E}(t)}\kappa_i\kappa_j\cr
&=\sum_{i=1,2,3}\kappa^2_i+C\sqrt{\mathcal{E}(t)}\Big\{\sum_{1\leq i\leq 3}\kappa^2_i+\sum_{1\leq i,j\leq3}\kappa_i\kappa_j\Big\}\cr
&\leq \left\{1+C_1\sqrt{\mathcal{E}(t)}\right\}|\kappa|^2.
\end{align*}
Likewise,
\begin{align*}
\kappa^{\top} \mathcal{T}_{\nu,\theta\eta}\kappa\geq  \left\{1-C_2\sqrt{\mathcal{E}(t)}\right\}|\kappa|^2.
\end{align*}
This completes the proof of the claim.\newline

\noindent(1) Let $\{\lambda_i\}$ denote the eigenvalues of $\mathcal{T}_{\nu,\theta\eta}$ so that we can write
\begin{align*}
\mathcal{T}_{\nu,\theta\eta}=P^{\top}diag\{\lambda_1,\cdots,\lambda_{11}\}P,
\end{align*}
for some orthogonal matrix $P$. Here $diag\{a,b,\cdots\}$ denotes the diagonal matrix whose diagonal entries are $a,b,\cdots$.
Therefore, since the above claim implies
\[
1-C_1\mathcal{E}(t)\leq\lambda_i\leq 1+C_2\mathcal{E}(t)   \qquad(1\leq i\leq 11)
\]
for sufficiently small $\mathcal{E}(t)$, we have
\begin{align*}
X^{\top}\mathcal{T}^{-1}_{\nu,\theta \eta}Y&=X^{\top}\left[P^{\top}diag\{\lambda^{-1}_1,\cdots,\lambda^{-1}_{11}\}P\right]Y\cr
&=\{PX\}^{\top}diag\{\lambda^{-1}_1,\cdots,\lambda^{-1}_{11}\}\{PY\}\cr
&= \sum_i\lambda^{-1}_i\{PX\}_i\{PY\}_i\cr
&\leq \max\{\lambda^{-1}_i\}\|PX\|\|PY\|\cr
&= \max\{\lambda^{-1}_i\}\|X\|\|Y\|.\cr
&\leq \{1-C_1\mathcal{E}(t)\}^{-1}\|X\|\|Y\|.
\end{align*}

\noindent (2) Lower bound can be computed in a similar way as follows:
\begin{align*}
X^{\top}\mathcal{T}^{-1}_{\nu,\theta \eta}X&=X^{\top}P^{\top}diag\{\lambda^{-1}_1,\cdots,\lambda^{-1}_{11}\}PX\cr
&= \sum_i\lambda^{-1}_i\left|\{PX\}_i\right|^2\cr
&\geq \min\{\lambda^{-1}_i\}\|PX\|^2\cr
&= \min\{\lambda^{-1}_i\}\|X\|^2\cr
&\geq  \{1+C_2\mathcal{E}(t)\}^{-1}\|X\|^2.
\end{align*}
\end{proof}
In the next lemma, we prove an estimate for derivatives of $\mathcal{T}_{\nu,\theta\eta}$.
%
%
%
%
\begin{lemma}\label{UniformEst2}
Let $0\leq \theta\leq 1$ and $-1/2<\nu<1$. Assume $\mathcal{E}(t)$ is sufficiently small. Then we have
\begin{eqnarray*}
~\|\partial^{\alpha}\big(\mathcal{T}^{-1}_{\nu,\theta\eta}\big)\|\leq C_{\alpha}\sqrt{\mathcal{E}(t)}.
\end{eqnarray*}
\end{lemma}
\begin{proof}
Applying $\partial\left\{\mathcal{T}_{\nu,\theta,\eta}^{-1}\right\}
=-\mathcal{T}_{\nu,\theta,\eta}^{-1}\left\{\partial\mathcal{T}_{\nu,\theta,\eta}\right\}\mathcal{T}_{\nu,\theta,\eta}^{-1}$ recursively, we can derive
$\partial^{\alpha}\left\{\mathcal{T}_{\nu,\theta,\eta}^{-1}\right\}=
P\big(\mathcal{T}_{\nu,\theta,\eta}^{-1},\partial \mathcal{T}_{\nu,\theta,\eta},\cdots,\partial^{\alpha}\mathcal{T}_{\nu,\theta,\eta}\big)$ for some polynomial $P$. Therefore, by Lemma \ref{ULestimatesofRhoUT} (4), (5) and Lemma \ref{ULestimatesofRhoUT1} (3), we get
the desired result.
\end{proof}

\section{Estimates on nonlinear perturbation and Local existence}

%
%
%
%
In this section, we establish the local in time existence of smooth solutions. For this, we first need to estimate the nonlinear part.
%
%
%
%
%
%
%
%
\begin{lemma}\label{GammaEstimate} The nonlinear perturbation $\Gamma_{\nu,\theta}(f)$ satisfies:
\begin{align*}
&(1)~ \left|~\int \partial^{\alpha}_{\beta}\Gamma_{\nu,\theta}(f)gdvdI\right|\leq C\!\!\!\!\!\sum_{|\alpha_1|+|\alpha_2|\leq|\alpha|}
\|\partial^{\alpha_1}f\|_{L^2_{v,I}}\|\partial^{\alpha_2}f\|_{L^2_{v,I}}\|g\|_{L^2_{v,I}}\cr
&\hspace{3.8cm}+ C\!\!\!\!\!\sum_{\substack{|\alpha_1|+|\alpha_2|\leq|\alpha|\\|\beta_1|\leq|\beta|}}
\|\partial^{\alpha_1} f\|_{L^2_{v,I}}\|\partial^{\alpha_2}_{\beta_1} f\|_{L^2_{v,I}}\|g\|_{L^2_{v,I}}\cr
&\hspace{3.8cm}+C\!\!\!\!\!\sum_{|\alpha_1|+|\alpha_2|+|\alpha_3|\leq|\alpha|}
\|\partial^{\alpha_1}f\|_{L^2_{v,I}}\|\partial^{\alpha_2}f\|_{L^2_{v,I}}
\|\partial^{\alpha_3}f\|_{L^2_{v,I}}\|h\|_{L^2_{v,I}},\cr
&(2)~ \left\|\int\Gamma_{1,2}(f,g)hdvdI\right\|_{L^2_{x}}\!\!\!+\left\|\int\Gamma_{1,2}(g,h)hdvdI\right\|_{L^2_{x}}\leq C\sup_{x,v,I}|\nu_{v,I}h|\sup_x\|f\|_{L^2_{v,I}}\|g\|_{L^2_{x,v,I}},\cr
&\hspace{0.6cm} \left\|\int\Gamma_3(f,g,h)rdvdI\right\|_{L^2_{x}}\!\!+\left\|\int\Gamma_3(g,f,h)rdvdI\right\|_{L^2_{x}}\!\!
+\left\|\int\Gamma_3(g,h,f)rdvdI\right\|_{L^2_{x}}\cr
&\hspace{1.43cm}\leq C\sup_{x,v,I}|\nu_{v,I}r|\sup_x\|f\|_{L^2_{v,I}}\sup_x\|g\|_{L^2_{v,I}}\|h\|_{L^2_{x,v,I}},
\end{align*}
where $\nu_{v,I}=(1+|v|^2)(1+I)$.
\end{lemma}
\begin{proof}
We prove this lemma only for $\Gamma_2$, Other terms can be treated similarly.
We first need to estimate $\partial^{\alpha}_{\beta}Q^{\mathcal{M}}_{ij}$. \newline
\noindent{\bf Claim:} $\displaystyle\big|\partial^{\alpha}_{\beta}Q^{\mathcal{M}}_{ij}\big|\leq
 C_{\alpha,\beta}
\exp\left(-\frac{|v|^2}{12}-\frac{I^{2/\delta}}{6}\right)$.\newline
\noindent {\bf Proof of the claim:} Note that there exist a homogeneous polynomial $P_{\alpha,\beta}$ and a monomial $M_{\alpha,\beta}$ such that
\begin{align*}
\left|\partial^{\alpha}_{\beta}\mathcal{M}_{\nu,\theta}\left(\rho_{\eta},U_{\eta},\mathcal{T}_{\nu,\theta\eta},T_{\theta\eta},\right)\right|
&=\frac{\left|P_{\alpha,\beta}\left(\partial\rho_{\eta},\partial U_{\eta},\partial (v-U_{\eta}), \partial\mathcal{T}_{\nu,\theta\eta}, \partial T_{\theta\eta},\partial I^{2/\delta}\right)\right|}{M_{\alpha,\beta}\left(\det(\mathcal{T}_{\nu,\theta\eta}), T_{\theta\eta}\right)}\cr
&\times\exp\left(-\frac{1}{2}(v-U_{\eta})^{\top}\mathcal{T}^{-1}_{\nu,\theta\eta}(v-U_{\eta})-\frac{I^{2/\delta}}{T_{\theta\eta}}\right).
\end{align*}
Here we slightly abused the notation to let  $\partial$ denote any of $\partial^{\bar{\alpha}}_{\bar{\beta}}$ such that $\bar{\alpha}\leq|\alpha|$ and $\bar{\beta}\leq|\beta|$.
Recalling the upper and lower bound estimates on the macroscopic fields in Lemma \ref{ULestimatesofRhoUT},
Lemma \ref{ULestimatesofRhoUT1}, the determinant estimates in Lemma \ref{Det_Estimate} and the estimates on the temperature tensor made in Lemma \ref{UniformEst}, Lemma \ref{UniformEst2}, we have
\begin{align*}
\frac{\left|P_{\alpha,\beta}\left(\partial\rho_{\eta},\partial U_{\eta},\partial (v-U_{\eta}), \partial\mathcal{T}_{\nu,\theta\eta}, \partial T_{\theta\eta}, I^{2\delta}\right)\right|}{M_{\alpha,\beta}\left(\det(\mathcal{T}_{\nu,\theta\eta}),T_{\theta\eta}\right)}
\leq C_{\alpha,\beta}(1+|v|^2+I^{2/\delta})^{m(\alpha)}.
\end{align*}
for some $C_{\alpha,\beta}, m(\alpha)>0$.
On the other hand, Lemma \ref{ULestimatesofRhoUT} (2), (5) and Lemma \ref{UniformEst} (2)  give the following lower bound:
\begin{align*}
\frac{1}{2}(v-U_{\eta})^{\top}\mathcal{T}^{-1}_{\nu,\theta\eta}(v-U_{\eta})+\frac{I^{2/\delta}}{T_{\theta\eta}}
&\geq \bigg(\frac{2}{3}+2\varepsilon\bigg)\frac{|v-U_{\eta}|^2}{2}+\bigg(\frac{2}{3}+\varepsilon\bigg)I^{2/\delta}\cr
&\geq \bigg(\frac{1}{3}+\varepsilon\bigg)|v|^2+\bigg(\frac{2}{3}+\varepsilon\bigg)I^{2/\delta}+o\big(\mathcal{E}(t)\big),
\end{align*}
for some small $\varepsilon>0$. This gives
\begin{align}\label{Q_Estimate}
\begin{split}
\left|\partial^{\alpha}_{\beta}Q^{\mathcal{M}}_{ij}\right|
&\leq
C_{\alpha,\beta}\exp\left(\frac{1}{4}|v|^2+\frac{I^{2/\delta}}{2}\right)(1+|v|^2+I^{2/\delta})^{m(\alpha)}\cr
&\times\exp\left(-\bigg(\frac{1}{3}+\varepsilon\bigg)|v|^2
-\bigg(\frac{2}{3}+\varepsilon\bigg)I^{2/\delta}+o\big(\mathcal{E}(t)\big)\right)\cr
&\leq C_{\alpha,\beta}\exp\left(-\frac{|v|^2}{12}-\frac{I^{2/\delta}}{6}\right).
\end{split}
\end{align}
This completes the proof of the claim. We now return to the proof of the lemma.
In view of (\ref{Q_Estimate}), we denote throughout this proof
\[
\widetilde{m}(v,I)=\exp\left(-\frac{|v|^2}{12}-\frac{I^{2/\delta}}{6}\right)
\]
for simplicity.\newline
\noindent(1) From (\ref{Q_Estimate}) and H\"{o}lder inequality, we have
\begin{align*}
&\int_{\mathbb{R}^3\times\mathbb{R}_+}|\partial^{\alpha}_{\beta}\Gamma_2(f)g|dvdI\cr
&\qquad\leq \sum_{\substack{|\alpha_1|+|\alpha_2|+|\alpha_3|\\=|\alpha|}}
\int_{\mathbb{R}^3\times\mathbb{R}_+}\big|\partial^{\alpha_1}_{\beta_1}Q^{\mathcal{M}}_{ij}
\langle \partial^{\alpha_2}f,e_i\rangle_{L^2_{v,I}}\langle \partial^{\alpha_2}f,e_j\rangle_{L^2_{v,I}}\,g\big|dvdI\cr
&\qquad\leq C\sum_{\substack{|\alpha_1|+|\alpha_2|+|\alpha_3|\\=|\alpha|}} \bigg(\int_{\mathbb{R}^3\times\mathbb{R}_+}\big|\partial^{\alpha_1}_{\beta_1}Q^{\mathcal{M}}_{ij}\,\,g\big|dvdI\bigg)\,
\| \partial^{\alpha_2}f\|_{L^2_{v,I}}
\|\partial^{\alpha_3}f\|_{L^2_{v,I}}\cr
&\qquad\leq C_{\alpha,\beta}\sum_{\substack{|\alpha_1|+|\alpha_2|+|\alpha_3|\\=|\alpha|}}
\bigg(\int_{\mathbb{R}^3\times\mathbb{R}_+}\widetilde{m}\,|g|\,dvdI\bigg)\,
\| \partial^{\alpha_2}f\|_{L^2_{v,I}}
\|\partial^{\alpha_3}f\|_{L^2_{v,I}}\cr
&\qquad\leq C_{\alpha,\beta}\sum_{\substack{|\alpha_1|+|\alpha_2|+|\alpha_3|\\=|\alpha|}}
\| \partial^{\alpha_2}f\|_{L^2_{v,I}}
\|\partial^{\alpha_3}f\|_{L^2_{v,I}}\|g\|_{L^2_{v,I}}\cr
&\qquad\leq C_{\alpha,\beta}\sum_{|\alpha_1|+|\alpha_2|\leq|\alpha|}
\| \partial^{\alpha_1}f\|_{L^2_{v,I}}
\|\partial^{\alpha_2}f\|_{L^2_{v,I}}\|g\|_{L^2_{v,I}}.
\end{align*}
Here we omitted $\sum_{ij}$ for simplicity of presentation.\newline
\noindent(2) Note that when $\alpha=\beta=0$,
we have much simpler estimate: $|Q^{\mathcal{M}}_{ij}|\leq C\widetilde{m}$ directly from Lemma \ref{ULestimatesofRhoUT}, Lemma
\ref{Det_Estimate} (2) and Lemma \ref{UniformEst} (2). Therefore, applying H\"{o}ler inequality, we obtain
\begin{align*}
\int_{\mathbb{R}^3\times\mathbb{R}_+}\Gamma_2(f,g)hdvdI&\leq C\int_{\mathbb{R}^3}\|f\|_{L^2_{v,I}}\|g\|_{L^2_{v,I}}
\bigg(\int_{\mathbb{R}^3\times\mathbb{R}_+}\widetilde{m} |h|dvdI\bigg)\cr
&\leq C\|f\|_{L^2_{v,I}}\|g\|_{L^2_{v,I}}\|h\|_{L^2_{v,I}}\cr
&\leq C\sup_{v,I}|\nu_{v,I}h||\sup_x\|f\|_{L^2_{v,I}}\|g\|_{L^2_{v,I}}.
\end{align*}
We then take $L^2_x$ norm to get the result. The proof for other terms are similar.
\end{proof}
\subsection{Local existence:}
Now, the local existence theorem can be proved by standard arguments (See, e.g \cite{Guo VMB}).
\begin{theorem}\label{localExistence}
Let $0\leq \theta\leq 1$ and $-1/2<\nu<1$. Let $F_0=m+\sqrt{m }f_0\geq 0$ and $f_0$ satisfies (\ref{ConservationLawsf}).
Then there exist $M_0>0$, $T_*>0$, such that if $ \mathcal{E}(0)\leq\frac{M_0}{2}$, then there is a unique solution $f(t,x,v,I)$ to (\ref{LBGK}) defined on $[0,T_*)$, such that
\begin{enumerate}
\item The high order energy $\mathcal{E}\big(f(t)\big)$ is continuous in $[0,T*)$ and uniformly bounded:
\begin{eqnarray*}
\sup_{0\leq t\leq T_*}\mathcal{E}\big(f(t)\big)\leq  M_0.
\end{eqnarray*}
\item The distribution function stays non-negative on $[0,T_*)$:
\begin{eqnarray*}
F(t,x,v,I)=m +\sqrt{m }f(t,x,v,I)\geq 0.
\end{eqnarray*}
\item The conservation laws (\ref{ConservationLawsf}) hold for all $[0,T_*)$.
\end{enumerate}
\end{theorem}
\begin{proof}
We consider the following scheme:
\begin{eqnarray}\label{Localscheme}
\partial_tF^{n+1}+v\cdot\nabla_xF^{n+1}=\frac{\rho^n T_{\delta}^n}{1-\nu+\theta\nu}\left\{\mathcal{M}_{\nu,\theta}(F^n)-F^{n+1}\right\},
\end{eqnarray}
with
\begin{eqnarray*}
\mathcal{M}_{\nu,\theta}(F^n)=\frac{\rho^n\Lambda_{\delta}}{\sqrt{\det(2\pi\mathcal{T}^n_{\nu,\theta})}\{T^{n}_{\theta}\}^{\delta/2}}\exp\left(\frac{1}{2}(v-U^n)\left\{\mathcal{T}^{n}_{\nu,\theta }\right\}^{-1}(v-U^n)-\frac{I^{2/\delta}}{T^n_{\theta}}\right),
\end{eqnarray*}
where $\rho^n$, $U^n$, $\mathcal{T}^n_{\nu,\theta}$ and $T^n_{\theta}$ denote the local density, bulk velocity and the temperature tensor associated with $F^n=m +\sqrt{m }f^n$.
Making use of Lemma \ref{GammaEstimate}, the local existence follows from a standard argument.
(See \cite{Guo VMB}).
The only difference from the usual proof is that the strict positiveness of $\mathcal{T}^n_{\nu,\theta}$ and $T^n_{\delta}$ should be secured in each step, so that $\mathcal{M}_{\nu,\theta}(F^n)$ is well-defined. This is guaranteed by Lemma \ref{ULestimatesofRhoUT} (5) and Lemma \ref{UniformEst}.
\end{proof}
%
%
%
%
\section{Micro-macro system} In this section, we study the micro-macro system of (\ref{LBGK}) to fill up the degeneracy in the dissipation estimates in Theorem \ref{degenerate coercivity}. The dichotomy in the dissipation
estimate observed in Theorem \ref{degenerate coercivity} indicates that we should employ two different sets of micro-macro decomposition.
\subsection{Micro-macro system I {\bf($0<\theta\leq 1)$}:}
Define
\begin{align*}
a(x,t)&=\int_{\mathbb{R}^3_v\times \mathbb{R}^+_I}f\sqrt{m }dvdI,\cr
b_i(x,t)&=\int_{\mathbb{R}^3_v\times \mathbb{R}^+_I}fv_i\sqrt{m }dvdI, \quad(i=1,2,3)\cr
c(x,t)&=\int_{\mathbb{R}^3_v\times \mathbb{R}^+_I}f\left(\frac{(|v|^2-3)+(2I^{2/\delta}-\delta)}{\sqrt{2(3+\delta)}}\right)\sqrt{m }dvdI,
\end{align*}
so that the polyatomic projection operator $P_p$ is written
\begin{eqnarray*}
P_pf= a(x,t)\sqrt{m }+\sum_{i}b_i(x,t)v_i\sqrt{m }+c(x,t)\left(\frac{(|v|^2-3)+(2I^{2/\delta}-\delta)}{\sqrt{2(3+\delta)}}\right)\sqrt{m }.
\end{eqnarray*}
Since $L_{\nu,\theta}\{P_pf\}=0$ for  $0<\theta\leq1$ by Proposition \ref{Kernel L},
the linearized polyatomic BGK model (\ref{LBGK}) is decomposed into the macroscopic part and the microscopic parts as follows:
\begin{eqnarray*}
\{\partial_t+v\cdot\nabla_x\}\{P_pf\}=-\{\partial_t+v\cdot\nabla_x\}\{(I-P_p)f\}+L_{\nu,\theta}\{(I-P_p)f\}+\Gamma_{\nu,\theta}(f).\
\end{eqnarray*}
%
%
%
%
%
We then expand the l.h.s and r.h.s with respect to the following basis $(1\leq i,j\leq 3)$:
\begin{eqnarray}\label{basis}
\big\{
\sqrt{m },\,v_i\sqrt{m }, \,v_iv_j\sqrt{m }, \,v_i^2\sqrt{m }, \,v_i\big\{ |v|^2+(2I^{2/\delta}-\delta)\big\}\sqrt{m },
\,( 2I^{2/\delta}-\delta )\sqrt{m }
\big\}.
\end{eqnarray}
Comparing the coefficients on both sides, we derive the following micro-macro system:
\begin{align}\label{MicroMacro}
\begin{split}
\partial_t  a-3A_{\delta}\partial_tc&=\ell_{a}+h_{a},\\
\partial_tb_i+\partial_{x_i} a-3A_{\delta}\partial_{x_i}c&=\ell_{abi}+h_{abi},\\
\partial_{x_i}b_j+\partial_{x_j}b_i&=\ell_{ij}+h_{ij}\quad (i\neq j)\\
A_{\delta}\partial_tc+\partial_{x_i} b_i&=\ell_{bci}+h_{bci},\\
A_{\delta}\partial_{x_i}c&=\ell_{ci}+h_{ci},\\
A_{\delta}\partial_{t}c&=\ell_{ct}+h_{ct},
\end{split}
\end{align}
for $i,j =1,2,3$. Here, $A_{\delta}=1/\sqrt{2(3+\delta)}$, and
$\ell_{a}$, $\ell_{abci}$, $\ell_{ij}$, $\ell_{bci}$, $\ell_{ci}$ and $\ell_{ct}$ denote the coefficients of projection of
$-\{\partial_t+v\cdot\nabla_x\}\{(I-P_m)f\}+L_{\nu,\theta}\{(I-P_m)f\}$ onto the basis (\ref{basis}),
and $h_{a}$, $h_{abci}$, $h_{ij}$, $h_{bci}$, $h_{ci}$ and $h_{ct}$ are the projection of $\Gamma_{\nu,\theta}(f)$ onto (\ref{basis}).
Adding the last two equations to the first and second line, we get
\begin{align}\label{MicroMacro}
\partial_t  a&=(\ell_{a}+h_{ct})+3\{\ell_{ct}+h_{ct}\},\nonumber\\
\partial_tb_i+\partial_{x_i}a&=\big\{\ell_{abi}+h_{abi}\big\}+3\{\ell_{ci}+h_{ci}\},\nonumber\\
\partial_{x_i}b_j+\partial_{x_j}b_i&=\ell_{ij}+h_{ij}\quad (i\neq j)\\
A_{\delta}\partial_{t}c+\partial_{x_i} b_i&=\ell_{bci}+h_{bci},\nonumber\\
A_{\delta}\partial_{x_i}c&=\ell_{ci}+h_{ci},\nonumber\\
A_{\delta}\partial_t c&=\ell_{ct}+h_{ct},\nonumber.
\end{align}
The first 5 lines are, up to constant multiplication on the l.h.s and
additional slight complication in r.h.s, identical to the micro-macro system derived in \cite{Guo whole,Guo VMB, Yun} and the last line is easy to estimate.
Hence, following the same line of argument, we arrive at
\begin{align}\label{MicroMacroEstimatee}
\sum_{|\alpha|\leq N}\left\{\|\partial^{\alpha}a\|^2_{L^2_{x}}+\|\partial^{\alpha}b\|^2_{L^2_{x}}
+C_{\delta}\|\partial^{\alpha}c\|^2_{L^2_{x}}\right\}
\leq C \sum_{|\alpha|\leq N-1}\|\partial^{\alpha}(\ell_{\nu,\theta}+h_{\nu,\theta})\|^2_{L^2_{x}}
\end{align}
Note that we have slightly abused the notation to denote
$\ell_{\nu,\theta}=(\ell_{a}, \ell_{abi}$, $\ell_{ij}, \ell_{bci}, \ell_{ci},\ell_{ct})$ and $h_{\nu,\theta}=(h_{a}, h_{abci}, h_{ij}$, $h_{bci}, h_{ci}, h_{ct})$.
Now, $\ell_{\nu,\theta}$ and $h_{\nu,\theta}$ can be controlled in a standard way as follows:
%
%
%
%
\begin{eqnarray*}
\sum_{|\alpha|\leq N-1}\|\partial^{\alpha}(\ell_{\nu,\theta}+h_{\nu,\theta})\|^2_{L^2_{x}}
\leq C\sum_{|\alpha|\leq N}\|(I-P_p)\partial^{\alpha}f\|^2_{L^2_{x,v,I}}+
C\sqrt{M_0}\sum_{|\alpha|\leq N}\|\partial^{\alpha}f\|^2_{L^2_{x,v,I}}.
\end{eqnarray*}
%
%
%
%
Combining this with (\ref{MicroMacroEstimatee}), we derive
\begin{align*}
\sum_{|\alpha|\leq N}\|\partial^{\alpha}P_pf\|^2_{L^2_{x,v,I}}
&\leq \sum_{|\alpha|\leq N}\left\{\|\partial^{\alpha}a\|^2_{L^2_{x,v,I}}+\|\partial^{\alpha}b\|^2_{L^2_{x,v,I}}+\|\partial^{\alpha}c\|^2_{L^2_{x,v,I}}\right\}\cr
&\leq C\sum_{|\alpha|\leq N}\|\partial^{\alpha}(I-P_p)f\|^2_{L^2_{x,v,I}}
+C\sqrt{M}_0\sum_{|\alpha|\leq N}\|\partial^{\alpha}f\|^2_{L^2_{x,v,I}},
\end{align*}
which gives
\begin{eqnarray}\label{microMacroEstimatee}
\sum_{|\alpha|\leq N}\|P_p\partial^{\alpha}f\|^2_{L^2_{x,v,I}}
\leq C\sum_{|\alpha|\leq N}\|(I-P_p)\partial^{\alpha}f\|^2_{L^2_{x,v,I}}.
\end{eqnarray}
Then, we conclude from Proposition \ref{degenerate coercivity} and (\ref{microMacroEstimatee}) that
there exists $C_{\nu,\theta}>0$ such that
\begin{eqnarray}\label{Coercivity}
\sum_{|\alpha|\leq N}\langle L_{\nu,\theta}\partial^{\alpha}f,\partial^{\alpha}f\rangle_{L^2_{x,v,I}}\leq -C_{\nu,\theta}\sum_{|\alpha|\leq N}\|\partial^{\alpha}f(t)\|^2_{L^2_{x,v,I}}
\end{eqnarray}
for sufficiently small $\mathcal{E}(t)$.

%
%
%
%
%
%
%
%
\subsection{Micro-macro system II~{\bf $(\theta=0)$:}}
Recalling Proposition \ref{Kernel L}, we see that in this case, (\ref{LBGK}) should be decomposed with respect to the monatomic-like projection $P_mf$.
In view of this observation, we define
\begin{align*}
a(x,t)&=\int_{\mathbb{R}^3_v\times \mathbb{R}^+_I}f\sqrt{m }dvdI,\cr
b_i(x,t)&=\int_{\mathbb{R}^3_v\times \mathbb{R}^+_I}fv_i\sqrt{m }dvdI, \quad(i=1,2,3)\cr
c(x,t)&=\int_{\mathbb{R}^3_v\times \mathbb{R}^+_I}f\left(\frac{|v|^2-3}{\sqrt{6}}\right)\sqrt{m }dvdI,\cr
d(x,t)&=\int_{\mathbb{R}^3_v\times \mathbb{R}^+_I}f\left(\frac{2I^{2/\delta}-\delta}{\sqrt{2\delta}}\right)\sqrt{m }dvdI,
\end{align*}
to write
\begin{align*}
P_mf&= a\sqrt{m }+\sum_{i}b_iv_i\sqrt{m }
+c\left(\frac{|v|^2-3}{\sqrt{6}}\right)\sqrt{m }
+d\left(\frac{2I^{2/\delta}-\delta}{\sqrt{2\delta}}\right)\sqrt{m }.
\end{align*}
We recall from  Proposition \ref{Kernel L} that $L_{\nu,0}(P_mf)=0$, and divide (\ref{LBGK}) into the macroscopic part and the microscopic part as follows:
\begin{eqnarray*}
\{\partial_t+v\cdot\nabla_x\}\{P_mf\}=-\{\partial_t+v\cdot\nabla_x\}\{(I-P_m)f\}+L_{\nu,0}\{(I-P_m)f\}+\Gamma_{\nu,0}(f).
\end{eqnarray*}
%
%
%
%
%
Comparing coefficients corresponding to the following basis:
%
%
%
%
%
\begin{eqnarray}\label{basis2}
\big\{
\sqrt{m},v_i\sqrt{m}, v_iv_j\sqrt{m}, v_i^2\sqrt{m}, v_i|v|^2\sqrt{m}, \big(2I^{2/\delta}-\delta\big)\sqrt{m},
\big(2I^{2/\delta}-\delta\big)v_i\sqrt{m}
\big\}.
\end{eqnarray}
We obtain $(i,j=1,2,3)$:
\begin{align}\label{MicroMacro0}
\partial_t a-3/\sqrt{6}\partial_tc&=\ell_{a}+h_{a},\nonumber\\
\partial_tb_i+\partial_{x_i}a-3/\sqrt{6}\partial_{x_i}c&=\ell_{abci}+h_{abci},\nonumber\\
\partial_{x_i}b_j+\partial_{x_j}b_i&=\ell_{ij}+h_{ij},\\
\partial_{x_i} b_i+1/\sqrt{6}\partial_{t}c&=\ell_{bci}+h_{bci},\nonumber\\
1/\sqrt{6}\partial_{x_i}c&=\ell_{ci}+h_{ci},\nonumber\\
\partial_td&=\ell_{dt}+h_{dt},\nonumber\\
\partial_{x_i}d&=\ell_{dxi}+h_{dxi},\nonumber
\end{align}
where $\ell_{a}$, $\ell_{abci}$, $\ell_{ij}$, $\ell_{bci}$, $\ell_{ci}$, $\ell_{dt}$ and $\ell_{dxi}$ are obtained by taking the inner product of
$-\{\partial_t+v\cdot\nabla_x\}\{(I-P_m)f\}+L_{\nu,0}\{(I-P_m)f\}$ with the basis in (\ref{basis2}),
and $h_{a}$, $h_{abci}$, $h_{ij}$, $h_{bci}$, $h_{ci}$, $h_{dt}$ and $h_{dxi}$ are the inner product of $\Gamma_{\nu,0}(f)$ with (\ref{basis2}).
Setting $\tilde{a}=a-3/\sqrt{6}c$ and $\tilde{c}=1/\sqrt{6}c$, we derive from (\ref{MicroMacro0})
\begin{align}\label{MicroMacro}
\partial_t \tilde{a}&=\ell_{a}+h_{a},\nonumber\\
\partial_tb_i+\partial_{x_i}\tilde{a}&=\ell_{abi}+h_{abi},\nonumber\\
\partial_{x_i}b_j+\partial_{x_j}b_i&=\ell_{ij}+h_{ij},\\
\partial_{x_i} b_i+\partial_{t}\tilde{c}&=\ell_{bc1}+h_{bc1},\nonumber\\
\partial_{x_i}\tilde{c}&=\ell_{ci}+h_{ci},\nonumber\\
\partial_td&=\ell_{dt}+h_{dt},\nonumber\\
\partial_{x_i}d&=\ell_{dxi}+h_{dxi}.\nonumber
\end{align}
Except for the last two line, which are decoupled from the other equations, and therefore, estimated easily, this is identical to the micro-macro system
for the usual Boltzmann equation or BGK model. Therefore, we can derive
\begin{eqnarray*}
\sum_{|\alpha|\leq N}\left\{\|\partial^{\alpha}\widetilde{a}\|^2_{L^2_{x}}+\|\partial^{\alpha}b\|^2_{L^2_{x}}
+\|\partial^{\alpha}\widetilde{c}\|^2_{L^2_{x}}
+\|\partial^{\alpha}d\|^2_{L^2_{x}}\right\}
\leq C \sum_{|\alpha|\leq N-1}\|\partial^{\alpha}(\ell_{\nu,0}+h_{\nu,0})\|^2_{L^2_{x}}.
\end{eqnarray*}
Here  we used the simplified notation again: $\ell_{\nu,0}=(\ell_{a}, \ell_{abci}, \ell_{ij}, \ell_{bci}, \ell_{ci},\ell_{dt},\ell_{dxi})$ and $h_{\nu,0}=(h_{a}, h_{abci}, h_{ij}$, $h_{bci}, h_{ci},h_{dt},h_{dxi})$.
We now take  $\frac{3\sqrt{6}}{10}<\varepsilon^2<\frac{2}{\sqrt{6}}$ and set
\begin{eqnarray*}
C_{\varepsilon}=\min\left\{(1-\sqrt{6}\varepsilon^2), \,3/5-\sqrt{6}/(2\varepsilon^2)\right\}>0
\end{eqnarray*}
to obtain
\begin{eqnarray*}
C_{\varepsilon}\sum_{|\alpha|\leq N}\left\{\|\partial^{\alpha}a\|^2_{L^2_{x}}+\|\partial^{\alpha}b\|^2_{L^2_{x}}+\|\partial^{\alpha}c\|^2_{L^2_{x}}
+\|\partial^{\alpha}d\|^2_{L^2_x}\right\}
\leq C \sum_{|\alpha|\leq N-1}\|\partial^{\alpha}(\ell_{\nu,0}+h_{\nu,0})\|^2_{L^2_{x}}.
\end{eqnarray*}
Then, by a similar argument as in the previous case, we can control $P_mf$ by $(I-P_m)f$:
\begin{eqnarray}\label{microMacroEstimate}
\sum_{|\alpha|\leq N}\|P_m\partial^{\alpha}f\|^2_{L^2_{x,v,I}}
\leq C\sum_{|\alpha|\leq N}\|(I-P_m)\partial^{\alpha}f\|^2_{L^2_{x,v,I}}.
\end{eqnarray}
which, combined with the dissipation estimate in Theorem \ref{degenerate coercivity} (2), implies
\begin{eqnarray}\label{Coercivity}
\sum_{|\alpha|\leq N}\langle L_{\nu,0}\partial^{\alpha}f,\partial^{\alpha}f\rangle_{L^2_{x,v,I}}\leq -C_{\nu}\sum_{|\alpha|\leq N}\|\partial^{\alpha}f(t)\|^2_{L^2_{x,v,I}}.
\end{eqnarray}

%
%
%
%
%
%
%
%

\section{Proof of Theorem \ref{main theorem}}
%
%
%
%
We have derived all the necessary estimates to close the energy estimate.
Let $f$ be the smooth local in time solution obtained in Theorem \ref{localExistence}.
Take derivatives on $x$, $t$ and $I$ of (\ref{LBGK}):
\begin{eqnarray*}
\partial_t \partial^{\alpha}f+v\cdot\nabla_x \partial^{\alpha}f=L_{\nu,\theta}\partial^{\alpha}f+\partial^{\alpha}\Gamma_{\nu,\theta}(f),
\end{eqnarray*}
and take inner product with $\partial^{\alpha}f$ to get
\begin{eqnarray*}
\frac{1}{2}\frac{d}{dt}\|\partial^{\alpha}f\|^2_{L^2_{x,v,I}}\leq \langle L_{\nu,\theta}\partial^{\alpha}f,\partial^{\alpha}f\rangle_{L^2_{x,v,I}}+\langle\partial^{\alpha}\Gamma_{\nu,\theta}(f),\partial^{\alpha}f\rangle_{L^2_{x,v,I}}.
\end{eqnarray*}
Making use of the coercivity estimate in the previous section yields
\begin{eqnarray*}
E^{\alpha}_0:\quad\frac{1}{2}\frac{d}{dt}\|\partial^{\alpha}f\|^2_{L^2_{x,v,I}}\!\!+C\sum_{|\alpha|\leq N}\|\partial^{\alpha}f\|^2_{L^2_{x,v,I}}\leq
C\sqrt{\mathcal{E}(t)}\mathcal{D}(t).
\end{eqnarray*}
For the energy estimate involving velocity derivatives, we apply $\partial^{\alpha}_{\beta}$ to (\ref{LBGK})
\begin{eqnarray*}
\big\{\partial_t+v\cdot\nabla_x +a_{\nu,\theta}\big\}\partial^{\alpha}_{\beta}f
=\sum_{|\beta_1|\neq0}\partial_{\beta_1}v\cdot\nabla_x\partial^{\alpha}_{\beta-\beta_1}f
+\partial_{\beta}P_{\nu,\theta}\partial^{\alpha}f+\partial^{\alpha}_{\beta}\Gamma_{\nu,\theta}(f,f),
\end{eqnarray*}
where $a_{\nu,\theta}=1/(1-\nu+\nu\theta)$. Then,
take inner product with $\partial^{\alpha}_{\beta}f$ and use H\"{o}lder inequality with Lemma \ref{GammaEstimate} to derive
\begin{eqnarray*}
\begin{split}
E^{\alpha}_{\beta}:\quad\frac{1}{2}\frac{d}{dt}\|\partial^{\alpha}_{\beta}f\|^2_{L^2_{x,v,I}}+a_{\nu,\theta}\|\partial^{\alpha}_{\beta}f\|^2_{L^2_{x,v,I}}
&\leq C\sum_{i}\|\partial^{\alpha+e_i}_{\beta-e_i}f\|_{L^2_{x,v,I}}\|\partial^{\alpha}_{\beta}f\|_{L^2_{x,v,I}}\cr
&+C\|\partial^{\alpha}f\|_{L^2_{x,v,I}}\|\partial^{\alpha}_{\beta}f\|_{L^2_{x,v,I}}+C\sqrt{\mathcal{E}(t)}\mathcal{D}(t),
\end{split}
\end{eqnarray*}
where $e_i (i=1,2,3)$ is the standard basis of $\mathbb{R}^3_x$. By Young's inequality, we can split the first two terms in the r.h.s as
\begin{align*}
\|\partial^{\alpha+e_i}_{\beta-e_i}f\|_{L^2_{x,v,I}}\|\partial^{\alpha}_{\beta}f\|_{L^2_{x,v,I}}&\leq
C_{\varepsilon}\|\partial^{\alpha+e_i}_{\beta-e_i}f\|^2_{L^2_{x,v,I}}+\varepsilon\|\partial^{\alpha}_{\beta}f\|^2_{L^2_{x,v,I}}\cr
\|\partial^{\alpha}f\|_{L^2_{x,v,I}}\|\partial^{\alpha}_{\beta}f\|_{L^2_{x,v,I}}&\leq
C_{\varepsilon}\|\partial^{\alpha}f\|^2_{L^2_{x,v,I}}+\varepsilon\|\partial^{\alpha}_{\beta}f\|^2_{L^2_{x,v,I}}
\end{align*}
whose $\varepsilon$ terms can be absorbed in the production term in the l.h.s to get
\begin{align*}
E^{\alpha}_{\beta}:\quad\frac{1}{2}\frac{d}{dt}\|\partial^{\alpha}_{\beta}f\|^2_{L^2_{x,v,I}}+\frac{1}{2}a_{\nu,\theta}\|\partial^{\alpha}_{\beta}f\|^2_{L^2_{x,v,I}}
&\leq C_{\varepsilon}\sum_{i}\|\partial^{\alpha+e_i}_{\beta-e_i}f\|^2_{L^2_{x,v,I}}+C_{\varepsilon}\|\partial^{\alpha}f\|^2_{L^2_{x,v,I}}\cr
&+ C\sqrt{\mathcal{E}(t)}\mathcal{D}(t).
\end{align*}
We then observe that the r.h.s of $\sum_{|\beta|=m+1}E^{\alpha}_{\beta}$ can be absorbed into the production terms in the
lower order estimate:
$C_{m}\sum_{|\beta|\leq m}E^{\alpha}_{\beta}$ for sufficiently large $C_m$. This observation enables one to find constants $C^1_m$, $C^2_m$  inductively such that
\begin{eqnarray*}
\sum_{\substack{|\alpha|+|\beta|\leq N,\cr|\beta|\leq m}}\left\{C^1_m\frac{d}{dt}\|\partial^{\alpha}_{\beta}f\|^2_{L^2_{x,v,I}}
+C^2_m\|\partial^{\alpha}_{\beta}f\|^2_{L^2_{x,v,I}}\right\}\leq C_N\sqrt{\mathcal{E}(t)}\mathcal{D}(t).
\end{eqnarray*}
Now, the standard continuity argument gives the global existence for (\ref{LBGK}) \cite{Guo VMB}.
This completes the proof.
\begin{align*}
\end{align*}



%
%
\bibliographystyle{amsplain}

\begin{thebibliography}{10}
\bibitem{ABLP} Andries, P., Bourgat, J.-F.,  Le Tallec, P., Perthame, B.: Numerical comparison between the Boltzmann and
ES-BGK models for rarefied gases. Comput. Methods Appl. Mech. Engrg. {\bf 191} (2002), no. 31, 3369-3390.
\bibitem{ALPP} Andries, P., Le Tallec, P., Perlat, J.-P., Perthame, B.: The Gaussian-BGK model of Boltzmann equation
with small Prandtl number. Eur. J. Mech. B Fluids {\bf 19} (2000), no. 6, 813-830.
\bibitem{Bello} Bellouquid, A.: Global existence and large-time behavior for BGK model for a gas with non-constant cross section, Transport Theory Statist. Phys. {\bf 32 } (2003) no. 2, 157-185.

\bibitem{BGK} Bhatnagar, P. L., Gross, E. P. and Krook, M.: A model for collision processes in gases. Small amplitude process in charged and neutral one-component systems, Physical Revies, {\bf 94} (1954), 511-525.
\bibitem{BL} Borgnakke. C.,  Larsen, P.S. : Statistical collision model for Monte Carlo simulation
of polyatomic gas mixture, J. Comput. Phys. {\bf18} (4) (1975) 405–420.
\bibitem{BDLP} Bourgat, J.-F., Desvillettes, L., Le Tallec, P., Perthame, B. : Microreversible collisions for polyatomic gases and Boltzmann's theorem. European J. Mech. B Fluids {\bf13} (1994), no. 2, 237–254.
\bibitem{B} Bouchut, F. : Construction of BGK models with a family of kinetic entropies for a given system of conservation laws. J. Statist. Phys. {\bf95} (1999), no.1-2, 113-170.
\bibitem{BP} Bouchut, F., Perthame, B.: A BGK model for small Prandtl number in the Navier-Stokes approximation. J. Stat. Phys. {\bf 71} (1993), no. 1-2, 191-207.
\bibitem{Brun} Brun R.: Transport et Relaxation dans les Écoulements Gazeux, Masson, 1986.
\bibitem{Brull2} Brull, S., Schneider, J. : A new approach of the Ellipsoidal Statistical Model. Cont. Mech. Thermodyn. 20 (2008), no.2, 63-74,
\bibitem{Brull3} Brull, S., Schneider, J. : On the ellipsoidal statistical model for polyatomic gases.
Contin. Mech. Thermodyn. {\bf20} (2009), no. 8, 489–508.
\bibitem{Cai}  Cai, Z., Li, R. : The NRxx method for polyatomic gases. J. Comput. Phys. {\bf267} (2014), 63–91.
\bibitem{Cercignani} Cercignani, C. : H-theorem and trend to equilibrium in the kinetic theory of gases.
Arch. Mech. (Arch. Mech. Stos.) {\bf34} (1982), no. 3, 231–241 (1983).

\bibitem{C} Cercignani, C.: The Boltzmann Equation and Its Application. Springer-Verlag, 1988.
\bibitem{CIP} Cercignani, C., Illner, R., Pulvirenti, M.: The Mathematical Theory of Dilute Gases. Springer-Verlag, 1994.
\bibitem{CC} Chapman, C. and Cowling, T. G.: The mathematical theory of non-uniform gases, Cambridge University Press, 1970.
\bibitem{Chan} Chan, W. M.: An energy method for the BGK model. M. Phil thesis, City University of Hong Kong, 2007.
\bibitem{DMOS} Dolbeault, J., Markowich, P., Oelz, D., Schmeiser, C. ; Non linear diffusions as limit of kinetic equations with relaxation collision kernels. Arch. Ration. Mech, Anal. {\bf186} (2007), no.1, 133-158.
\bibitem{DWY} Duan, R., Wang, Y.,  Yang, T. : Global Existence for the Ellipsoidal BGK Model with Initial Large Oscillations. Preprint: https://arxiv.org/abs/1607.01113
\bibitem{FJ} Filbet, F., Jin, S.: An asymptotic preserving scheme for the ES-BGK model of the Boltzmann equation. J. Sci. Comput. {\bf 46} (2011), no.2, 204-224.
\bibitem{GT} Galli, M.A., Torczynski, R.: Investigation of the ellipsoidal-statistical Bhatnagar-Gross-Krook kinetic model
applied to gas-phase transport of heat and tangential momentum between parallel walls, Phys. Fluids, {\bf 23} (2011) 030601
\bibitem{GL} Glassey, R.: The Cauchy Problems in Kinetic Theory. SIAM 1996.
\bibitem{GRS} Groppi, M.; Russo, G.; Stracquadanio, G. : High order semi-Lagrangian methods for the BGK equation. Commun. Math. Sci.  {\bf14}  (2016),  no. 2, 389–414.
\bibitem{GS} Groppi, M,  Spiga, G.: An ES–BGK Model for the Kinetic Analysis of a Chemically Reacting Gas Mixture. MATCH Commun. Math. Comput. Chem. {\bf69} (2013) 197--214
\bibitem{GS2} Groppi, M,  Spiga, G.: Kinetic approach to chemical reactions and inelastic transitions in a rarefied gas. Journal of Mathematical Chemistry  {\bf26} (1999) 197--219
\bibitem{Guo whole} Guo, Y.: The Boltzmann equation in the whole space. Indiana Univ. Math. J. {\bf 53} (2004). no.4, 1081-1094
\bibitem{Guo VMB} Guo, Y.: The Vlasov-Maxwell-Boltzmann system near Maxwellians. Invent. Math. {\bf 153} (2003) no.3, 593-630
\bibitem{Guo VPB} Guo, Y.: The Vlasov-Poisson-Boltzmann system near Maxwellians. Comm. Pure. Appl. Math., {\bf 55} (2002) no.9, 1104-1135.
\bibitem{Holway} Holway, L.H.: Kinetic theory of schock structure using and ellipsoidal distribution function. Rarefied Gas Dynamics, Vol. I(Proc. Fourth Internat. Sympos., Univ. Toronto, 1964), Academic Press, New York, (1966), pp. 193-215.
\bibitem{Issautier} Issautier, D.: Convergence of a weighted particle method for solving the Boltzmann (B.G.K.) equation, Siam Journal on Numerical Analysis, {\bf 33}, no 6 (1996), 2099-2199.
119–135.
\bibitem{KAG} Kosuge, S., Aoki, K., and Goto, T. :Shock wave structure in polyatomic gases: Numerical analysis using a model Boltzmann equation.
AIP Conference Proceedings 1786, 180004 (2016)
\bibitem{Krem} Kremer, G. M. :  An introduction to the Boltzmann equation and transport processes in gases. Interaction of Mechanics and Mathematics. Springer, Dordrecht, 2010.
\bibitem{LT} Lions, P. L. , Toscani, G. :Diffusive limit for finite velocity Boltzmann kinetic models.
Rev. Mat. Iberoamericana {\bf13} (1997), no. 3, 473-513.
\bibitem{MMM} Mellet, A., Mischler, S.,Mouhot,C. : Fractional diffusion limit for collisional kinetic equations. Arch. Ration. Mech. Anal. {\bf199} (2011),no.2 ,493-525.
\bibitem{Mellet}  Mellet, A.: Fractional diffusion limit for collisional kinetic equations: a moments method. Indiana Univ. Math. J.  {\bf59}  (2010),  no. 4, 1333–1360.
\bibitem{Mischler} Mischler, S.: Uniqueness for the BGK-equation in $\mathbb{R}^n$ and rate of convergence for a semi-discrete scheme. Differential integral Equations {\bf 9} (1996), no.5, 1119-1138.
\bibitem{PY} Cauchy problem for the ellipsoidal BGK model for polyatomic particles. Preprint:  https://arxiv.org/abs/1708.02461
\bibitem{PY1} Park, S., Yun, S.-B.: Cauchy problem for the ellipsoidal-BGK model of the Boltzmann equation. J. Math. Phys. 57 (2016), no. 8, 081512, 19 pp.
\bibitem{PY2} Park, S., Yun, S.-B. : Entropy production estimates for the polyatomic ellipsoidal BGK model. Appl. Math. Lett. {\bf58} (2016), 26–33.
\bibitem {PRS} Pavić, M., Ruggeri, T., Simić, S. : Maximum entropy principle for rarefied polyatomic gases. Phys. A {\bf392} (2013), no. 6, 1302–1317.
\bibitem{Perthame} Perthame, B. : Global existence to the BGK model of Boltzmann equation. J. Differential Equations. {\bf 82} (1989), no.1, 191-205.
\bibitem{PC} Coron, F., Perthame, B. : Numerical passage from kinetic to fluid equations. SIAM J. Numer. Anal.  {\bf28}  (1991),  no. 1, 26–42.
\bibitem{PPuppo} Pieraccini, S., Puppo, G.: Implicit-explicit schemes for BGK kinetic equations. J. Sci. Comput. {\bf 32} (2007), no.1, 1-28.
\bibitem{PL} Pitaevski L.P., Lifschitz E.M.: Physical Kinetics, Pergamon Press, Oxford, 1981.
\bibitem{PP} Perthame, B., Pulvirenti, M. : Weighted $L^{\infty}$ bounds and uniqueness for the Boltzmann BGK model. Arch. Rational Mech. Anal. {\bf 125} (1993), no. 3, 289-295.
\bibitem{PL} Pitaevski L.P., Lifschitz E.M.: Physical Kinetics, Pergamon Press, Oxford, 1981.
\bibitem{RS} Rahimi, B. , Struchtrup, H. : Capturing non-equilibrium phenomena in rarefied
polyatomic gases: A high-order macroscopic model. Physics of Fluids, {\bf 26}, (2014) 052001.
\bibitem{RS2} Ruggeri, T.; Sugiyama, M. : Recent developments in extended thermodynamics of dense and rarefied polyatomic gases. Acta Appl. Math. {\bf132} (2014), 527–548.
\bibitem{RSY} Russo, G., Santagati, P. and Yun, S.-B. : Convergence of a semi-Lagrangian scheme for the BGK model of the Boltzmann equation. arXiv: 1007.2843v1 [math.AP].
\bibitem{SR1} Saint-Raymond, L.: From the BGK model to the Navier-Stokes equations. Ann. Sci. Ecole Norm. Sup {\bf 36} (2003), no.2, 271-317.
\bibitem{SR2} Saint-Raymond, L.: Discrete time Navier-Stokes limit for the BGK Boltzmann equation. Comm. Partial Differential Equations {\bf 27} (2002), no. 1-2, 149-184.
\bibitem{Sone} Sone, Y.: Kinetic Theory and Fluid Mechanics. Boston: Birkh\"{a}user, 2002.
\bibitem{Sone2} Sone, Y.: Molecular Gas Dynamics: Theory, Techniques, and Applications. Boston: Brikh\"{a}user, 2006.
\bibitem{Stru} Struchtrup, H.: The BGK-model with velocity-dependent collision frequency. Contin. Mech. Thermodyn. {\bf 9} (1997), no.1 , 23-31.
\bibitem{Stru-book} Struchtrup, H.: Mesoscopic transport equations for rarefied gas flows: Approximation methods in kinetic theory. Springer. 2005.
\bibitem{TFA} Takata, S.; Funagane, H.; Aoki, K. Fluid modeling for the Knudsen compressor: case of polyatomic gases. Kinet. Relat. Models {\bf3} (2010), no. 2, 353–372.
\bibitem{Ukai} Ukai, S.: On the existence of global solutions of a mixed problem for the nonlinear Boltzmann equation. Proc. Japan Acad., Ser. A {\bf 53}, 179-184 (1974)
\bibitem{Ukai-BGK} Ukai, S.: Stationary solutions of the BGK model equation on a finite interval with large boundary data. Transport theory Statist. Phys. {\bf 21} (1991), no. 4-6. 487-500.
\bibitem{UT} Ukai, S. Yang, T.: Mathematical Theory of Boltzmann equation, Lecture Notes Series. no. 8, Liu Bie Ju Center for Math. Sci, City University of Hong Kong, 2006.
\bibitem{V} Villani, C.: A Review of mathematical topics in collisional kinetic theory. Handbook of mathematical fluid dynamics. Vol. I. North-Holland. Amsterdam, 2002, 71-305
\bibitem{Wel} Welander, P.: On the temperature jump in a rarefied gas, Ark, Fys. {\bf 7}  (1954), 507-553.
\bibitem{WZ} Wei, J., Zhang, X.: The Cauchy problem for the BGK equation with an external force.
J. Math. Anal. Appl. {\bf391} (2012), no. 1, 10–25.
\bibitem{Yun} Yun, S.-B. : Cauchy problem for the Boltzmann-BGK model near a global Maxwellian. J. Math. Phy. {\bf 51} (2010), no. 12, 123514, 24pp.
\bibitem{Yun2} Yun, S.-B.: Classical solutions for the ellipsoidal BGK model with fixed collision frequency. J. Differential Equations
    {\bf259} (2015), no. 11, 6009–6037.
\bibitem{Yun3} Yun, S.-B.: Ellipsoidal BGK model near a global Maxwellian. SIAM J. Math. Anal. {\bf47} (2015), no. 3, 2324–2354.
\bibitem{Yun4} Yun, S.-B.: Entropy production for ellipsoidal BGK model of the Boltzmann equation. Kinet. Relat. Models {\bf9} (2016), no. 3, 605–619.
\bibitem{Zhang} Zhang, X.: On the Cauchy problem of the Vlasov-Poisson-BGK system: global existence of weak solutions. J. Stat. Phys. {\bf141} (2010),no.3, 566-588.
\bibitem{ZH} Zhang, X., Hu, S.: $L^p$ solutions to the Cauchy problem of the BGK equation. J. Math. Phys. {\bf 48} (2007) no.11, 113304, 17pp.
\bibitem{Z-Stru} Zheng, Y., Struchtrup, H. : Ellipsoidal statistical Bhatnagar-Gross-Krook model with velocity dependent collision frequency. Phys. Fluids {\bf 17} (2005), 127103, 17pp.
\end{thebibliography}

\end{document}